\documentclass{article}

\usepackage{arxiv}

\usepackage[utf8]{inputenc} 
\usepackage[T1]{fontenc}    
\usepackage{url}            
\usepackage{booktabs}       
\usepackage{amsfonts}       
\usepackage{nicefrac}       
\usepackage{microtype}      
\usepackage{lipsum}

\usepackage{relsize,balance,lipsum,bbm,enumerate,times,comment,color,graphicx,setspace,mathdots,mathrsfs,amssymb,latexsym,amsfonts,amsmath,cite,stmaryrd,caption,pgf,accents,mathtools,tabu,enumitem,hhline,array,epstopdf,nicefrac,amsthm,microtype,algorithmic,array,float,bm,url}

\usepackage{graphicx}

\usepackage{mathtools} 

\usepackage[utf8]{inputenc}
\usepackage[english]{babel}
\newcommand{\an}[1]{{\textcolor{black}{#1}}}
\newcommand{\sa}[1]{{\textcolor{black}{#1}}}

\DeclareMathOperator*{\argmin}{arg\,min}
\usepackage{tikz}

\usetikzlibrary{positioning}
\definecolor{mygreen}{RGB}{153,255,153}
\definecolor{myorange}{RGB}{255,178,102}
\definecolor{myred}{RGB}{255,153,153}
\definecolor{myblue}{RGB}{153,204,255}

\newtheorem{thm}{Theorem}
\newtheorem{defn}{Definition}
\newtheorem{prop}{Proposition}
\newtheorem{lem}{Lemma}
\newtheorem{rem}{Remark}

\newtheorem{assum}{Assumption}
\newtheorem{exmp}{Example}
\def\conv{{\rm conv}}
\def\dom{{\rm dom}}

\def\an#1{{\color{black}#1}}
\def\sa#1{{\color{black}#1}}
\def\dist{{\rm dist}}

\title{On Fixed-Time Stability of Continuous Dynamics for Non-Monotone Variational Inequalities}

\author{
    Sina Arefizadeh$^*$\\
    School of Electrical and Computer Engineering\\
    Arizona State University\\
    Tempe, Arizona, USA\\
    \texttt{sarefiza@asu.edu}
\And
    Kunal~Garg\\
    School for Engineering of Matter, Transport and Energy\\
    Arizona State University\\
    Tempe, Arizona, USA\\
    \texttt{kgarg24@asu.edu}
\And
    Angelia~Nedi\'c\\
    School of Electrical and Computer Engineering\\
    Arizona State University\\
    Tempe, Arizona, USA\\
    \texttt{Angelia.Nedich@asu.edu}
}

\theoremstyle{definition}
\def\argmin{\mathop {\rm argmin}}

\begin{document}
\thanks{Corresponding author}
\thanks{The work has been partially supported by the NSF awards CCF-2106336 and CIF-2134256.}
\maketitle
\begin{abstract}
Non-monotone variational inequalities (NMVI) are an important class of problems that generalize non-convex optimization and have various applications in optimization theory, machine learning, game theory, and economics, among others. Most existing work on NMVIs focuses on the asymptotic convergence of algorithms proposed to solve these problems. 
In this paper, we tackle the problems of exponential and fixed-time stability of the solution set of a class of NMVIs for both unconstrained and constrained problems. 
We first present novel conditions guaranteeing exponential stability of 
solutions to unconstrained NMVIs for a uniquely constructed dynamical system under mild assumptions on the gradient of the non-monotone map. Then, under similar assumptions, we construct another novel dynamical system whose equilibrium point is fixed-time stable, i.e., the trajectories reach the equilibrium within a fixed time, independent of the initial conditions. For the case of constrained NMVIs, we employ a continuous-time variant of the Korpelevich method for exponential stability of the solution set, and provide a novel scaling factor in the dynamics to achieve fixed-time stability. We illustrate the efficacy of the proposed modified dynamical systems through numerical simulations and conclude the paper with a brief note on the behavior of the discretized variant of the proposed dynamics and on further work that remains to be done. 
\end{abstract}

\vspace{0.5cm} 


\section{Introduction}
The continuous-time dynamical systems approach has been well-studied for decades \cite{brown1989some}. More recent studies in this line of thought focus on developing techniques for discrete-time systems by mimicking them with their continuous-time system proxies \cite{wibisono2016variational,su2016differential}. A notable study that utilized ordinary differential equations (ODEs) for algorithmic development purposes is a continuous-time algorithm similar to the Nesterov accelerated method in the
context of control design\cite{krichene2015accelerated}. The superiority of this viewpoint lies in the availability of Lyapunov stability analysis for continuous-time systems. Extensive efforts were made to extend this perspective to constrained optimization \cite{brown1989some,wibisono2016variational,su2016differential,cherukuri2017saddle,dhingra2018proximal,osher2016sparse,durr2012class,krichene2015accelerated,helmke2012optimization,garg2020fixed,garg2022fixed}.

A question that has recently garnered attention is the fixed-time stability of solutions to continuous-time dynamical systems arising from optimization problems \cite{garg2020fixed}, where the convergence to a solution 
is achieved in a fixed time independent of the initial conditions. The origin of this problem dates back to \cite{polyakov2011nonlinear,bhat2000finite}, where the feedback design with fixed-time stability of the closed-loop system and finite-time stability of continuous autonomous systems was investigated. The study of gradient flow and its variants has been particularly employed in a variety of applications, ranging from image processing to motion planning \cite{clarenz2002relations,cortes2006finite}.

Ordinary differential equations and control systems have been used for optimization problems via gradient flow, and they have recently attracted interest for solving Variational Inequality (VI) problems~\cite{garg2022fixed,xia1998general,xia2002projection}. Studies \cite{bhat2000finite} and \cite{polyakov2011nonlinear} introduced the notions of finite-time and fixed-time stability of equilibrium for dynamical systems, respectively. An approximate discretization of these dynamics can be employed to tackle a similar problem in the discrete-time~\cite{garg2022fixed}. While the existing work explores the notion of faster convergence in the context of optimization and related problems \cite{cortes2006finite,garg2020fixed,poveda2021nonsmooth,poveda2022fixed,liu2023multiobjective}, most of the existing work assumes convexity or even stronger convexity conditions. There are works that deal with non-convex settings in 
optimization problems \cite{li2023convex,danilova2022recent}, or
the form of NMVI \cite{arefizadeh2024non,viet2025relaxed}.
However, these works deal with traditional notions of asymptotic and exponential convergence. 

In this paper, we study the fixed-time stability of the solution set of an NMVI through the lens of continuous-time dynamical systems. The foundations of our analysis are Lyapunov conditions for fixed-time stability \cite{polyakov2011nonlinear}. We consider a dynamic based on unconstrained minimization of the squared mapping norm for unconstrained NMVIs, and a continuous-time variant of the Korpelevich method for constrained NMVIs. 
Specifically, our main contributions in this paper are:\\
\noindent(1) 
{\it Unconstrained NMVI}. Based on minimizing the squared norm of the VI mapping, we propose a dynamic attracted to the solution set of the VI (Proposition~\ref{Prop-sol-set-attract}), relying on \cite{cortes2006finite} and  insights from our prior work~\cite{arefizadeh2024non,arefizadeh2025existence}. For this dynamic, we show that the solution set of the VI is exponentially stable (Theorem~\ref{Thm:Design-continuous-dynamic}). Then, we modify the dynamic so that the solution set of the VI is fixed-time stable (Theorem~\ref{Thm:Design-continuous-dynamic-fixed-time-stable}). These stability results are based on smooth Lyapunov analysis.\\
\noindent(2)
{\it Constrained NMVI}. We propose continuous-time dynamics for the discrete-time Korpelevich method
    attracted to the solution set of the VI. Under some assumptions on the VI mapping, we show that the solution set of the VI is exponentially stable (Theorem~\ref{thm-Korpelevich-analysis}). We then scale this dynamics with a novel scalar function so that the solution set of the VI is fixed-time stable (Theorem~\ref{thm-Korpelevich-v3-analysis}) for the modified dynamical system. Here, the stability results are based on non-smooth Lyapunov analysis. \\
\noindent (3)
{\it Alternative conditions on the mapping.} We provide alternatives to our basic assumption on the VI mapping that we have used in the fixed-time stability analysis of the VI solution set (Theorem~\ref{thm:smallest eigenvalue_Jacobian-direction-1}, Example~\ref{exmp-relation-mu-quasi-str-FxTS}, Example~\ref{exmp-relation-with-eigenvalues}).

The remainder of the paper is organized as follows: Section~\ref{Sec: Problem_Statement} provides the problem statement and some preliminary results. In Section~\ref{Sec:unconstrained}, the exponential and fixed-time stability results for the solutions of an unconstrained NMVI are presented under the dynamics that is aimed to minimize the squared norm of the mapping. In Section~\ref{sec:contrained}, the exponential and fixed-time stability results for the solutions of a constrained NMVI are established for continuous-time Korpelevich dynamics.
In Section~\ref{sec-disc}, we have a discussion on some alternative assumptions. 
In Section~\ref{Sec: Simulations}, numerical results are provided that illustrate the theoretical findings. Section~\ref{Sec: Conclusion} concludes the paper.

\section{Problem Statement and Preliminaries}\label{Sec: Problem_Statement}

\textit{Notations}: We use  $\|\cdot\|$ for the Euclidean norm and $\langle \cdot,\cdot\rangle$ for the inner product. The identity matrix is denoted by $\mathbb{I}$. 
For a set $X$ and a scalar $t$, the scaled set $tX$ is defined by
$tX=\{t x\mid x\in X\}$. The Minkowski sum of two sets $X,Y\subseteq \mathbb{R}^m$ is defined by $X+Y=\{x+y\mid x\in X, y\in Y\}$.
We use ${\rm conv}(X)$ for the convex hull of a set $X$, and ${\rm cl}(X)$ for the closure of the set $X$.
The distance from a point $x\in\mathbb{R}^m$ to a nonempty set $X\subset \mathbb{R}^m$ is given by $\dist(x,X)= \inf_{y\in X}\|x-y\|$. 
We use $\Pi_X[\cdot]$ to denote the Euclidean projection on a closed convex set $X\subseteq\mathbb{R}^m$, i.e., $\Pi_X[z]=\argmin_{x\in X}\|x-z\|$. 

For a locally Lipschitz function $V:\mathbb R^m\rightarrow \mathbb R$, the subdifferential set at a point $x\in \mathbb R^m$, denoted as $ \partial V(x)$, is given by $ \partial V(x)= \textnormal{conv}\{\lim_{k\to +\infty}\nabla V(x_k)\; |\; x_k\to x, x_k\notin Z\cup \Omega_V\}$, where $\Omega_V$ is the set of points where $V$ fails to be differentiable and $Z\subset\mathbb R^m$ is any set of Lebesgue measure zero.
The Lie derivative of a locally Lipschitz function $V$ along a continuous vector field $f:\mathbb R^m\rightarrow\mathbb R^m$ at a point $x\in \mathbb R^m$ is given by a set-valued map $\mathcal{L}_fV(x) = \{ \langle \xi,f(x)\rangle \mid \xi\in \partial V(x)\}$. 


We first introduce the VI problem of interest and provide some basic notions and preliminary results that we use later to design dynamics whose trajectories converge to the solutions of the VI (see, e.g., ~\cite{facchinei2003finite,konov,KSbook} for more details on VIs). 

\begin{defn}
Given a set $K\subseteq\mathbb{R}^m$ and a mapping $F:K\to\mathbb{R}^m$, the variational inequality problem, denoted by
VI$(K,F)$, consists of determining
   a point $x^*\in K$ such that 
   $$\langle F(x^*),x-x^*\rangle\geq 0\qquad\hbox{for all $x\in K$}.$$ 
   \end{defn}
A point $x^*$ satisfying the preceding inequality is referred to as a solution to the variational inequality problem VI$(K,F)$. The set of all such solutions is denoted by Sol$(K,F) \coloneq \{x^*\in K\; |\; \langle F(x^*),x-x^*\rangle\geq 0 \ \hbox{for all }x\in K\}$.

The mapping $F:K\to \mathbb{R}^m$ is monotone on the set $K$ when $\langle F(x)-F(y),x-y\rangle\geq 0$ for all $x,y\in K$. In this case, we say that the variational inequality VI$(K,F)$ is monotone. We are interested in non-monotone VI$(K,F)$,
where the mapping $F$ need not be monotone on the set $K$.

Our goal is to design a continuous-time nonlinear system of the form 
\begin{equation}\label{Eq:Sys}
    \dot{x}=f(x),
\end{equation}
where the mapping $f:\text{dom}(f)\to \mathbb{R}^m$, with $\text{dom}(f)\subseteq\mathbb R^m$ being the domain of the mapping $f$, is to be properly selected so that the trajectories of the system reach the solution set Sol$(K,F)$ of a non-monotone variational inequality problem VI$(K,F)$ exponentially fast. Also, we are interested in choosing the mapping $f$ such that any trajectory of the system reaches a solution or the solution set of the VI$(K,F)$ within a fixed time, independent of the initial state.

We assume that {\it the mapping $f$ is continuous} and that the system in~\eqref{Eq:Sys} {\it has a unique solution for every \an{initial point $x(0)\in\dom(f)$}}.
For a given initial point \an{$x(0)\in\dom(f)$}, the solution to the system in~\eqref{Eq:Sys} is a trajectory $x:[0,+\infty)\to\mathbb{R}^m$, which satisfies
\begin{equation}\label{eq:sys-solution}
    x(t)=x(0)+\int_{0}^t f\left(x(s)\right) ds\qquad\hbox{for all }t\ge0.
\end{equation}

We start with notions of invariant and positively invariant set for the nonlinear system in~\eqref{Eq:Sys}.

\begin{defn}\cite[pg.\ 127] {khalil2002nonlinear}\label{Def:Inv-Set} \textbf{Invariant and positively invariant set:}
    \an{A set $M\subseteq \dom(f)$} is said to be an invariant set with respect to \eqref{Eq:Sys} if \[x(0)\in M \quad \Rightarrow \quad x(t)\in M \quad \hbox{for all } t\in \mathbb R.\]
    The set $M$ is said to be a positively invariant set with respect to \eqref{Eq:Sys} if \[x(0)\in M \quad\Rightarrow \quad x(t)\in M \quad \hbox{for all }t\ge0.\]
\end{defn}
For a nonempty set $M\subseteq \mathbb{R}^m$, we let ${\rm dist}(x,M)$ be the distance function from $x$ to the set $M$, i.e.,
 \[{\rm dist}(x,M):=\inf_{y\in M}\|x-y\|.\] 
For a nonempty set $M \subseteq \mathbb{R}^m$ and $\delta>0$, we define
\[M_\delta:=\{x\in \mathbb{R}^m\mid {\rm dist}(x,M)<\delta\}.\]

\begin{defn}\cite[Definition~8.1]{khalil2002nonlinear}
\label{Def:stab-set}  
\textbf{Stable and asymptotically stable set}:
     A closed invariant set $M$ is stable for the system in~\eqref{Eq:Sys} if for each $\varepsilon>0$, there is $\delta>0$ such that 
     \[x(t)\in M_\varepsilon\qquad\hbox{for all $t\ge 0$ and all \an{$x(0)\in M_\delta\cap\dom(f)$}}.\]
    It is asymptotically stable if $\delta>0$ can be chosen such that
     \[\lim_{t\to\infty}\dist(x(t),M)=0 \quad\hbox{for all }x(0)\in \an{M_\delta\cap\dom(f)}.\]
\end{defn}

\begin{defn}\label{Def:expon-stab--2}  \textbf{Exponentially stable set}:
    A closed positively invariant set $M$ is exponentially stable for the nonlinear system \eqref{Eq:Sys} if there exist constants $\beta_1,\beta_2,\beta_3>0$ such that 
    for all \an{$x(0)\in \dom(f)$} with $\dist(x(0),M)\leq\beta_3$, it holds that
\begin{equation*}\label{eq:exp-stab--2}
       {\rm dist}(x(t),M)\leq \beta_1 {\rm dist}(x(0),M)e^{-\beta_2 t} \quad \hbox{for all } t\ge0.
    \end{equation*}
    The set $M$ is globally exponentially stable for system~\eqref{Eq:Sys} if the preceding inequality holds for all \an{$x(0)\in\dom(f)$}.
\end{defn}
Next, we present the definitions of finite-time stable and fixed-time stable closed sets.
\begin{defn} \label{def: FTS} \textbf{Finite-time stable (FTS) set}:
A closed positively invariant 
set $M$ is \emph{finite-time stable (FTS)} for the system in~\eqref{Eq:Sys} if $M$ is stable and there exists $\delta>0$ such that the trajectories of \eqref{Eq:Sys} approach the set $M$ in a finite time for every $x(0)\in M_\delta\cap\dom(f)\setminus M$, i.e., for any $x(0) \in M_\delta\cap\dom(f)\setminus M$, there is a finite time $T(x(0))>0$ such that
\[
\lim_{t \to T(x(0))} {\rm dist}\left(x(t),M\right) = 0.
\]
The time $T(x(0))$ is the settling time for the system in~\eqref{Eq:Sys} starting at the initial state $x(0)$.
The set $M$ is \emph{globally FTS stable} if $M_\delta = \dom(f)$.
\end{defn}
\begin{defn}\label{Def:FxTS}\textbf{Fixed-time stable (FxTS) set}:
A closed  and globally finite-time stable set $M$ is \emph{fixed-time stable (FxTS)} for the system in~\eqref{Eq:Sys} if the settling time $T(x(0))$ is bounded over the set \an{$\dom(f)$}, i.e.,  
\[\sup_{x(0)\in \an{\dom(f)}}T(x(0))<\infty.\]
\end{defn}

\begin{rem}
When the set $M$ is a singleton, e.g., there is a unique equilibrium for the system in~\eqref{Eq:Sys}, 
Definition~\ref{Def:stab-set} and Definition~\ref{Def:expon-stab--2} coincide with the standard definitions of stable, asymptotically stable, and (globally) exponentially stable equilibrium.
Definition~\ref{def: FTS} and Definition~\ref{Def:FxTS} are the same as Definition~3 and Definition~4 in~\cite{polyakov2011nonlinear}, respectively, where the term ``attractive" was used instead of ``stable".
\end{rem}

Next, we provide sufficient conditions for fixed-time stability of sets for the nonlinear system in~\eqref{Eq:Sys} using a potentially non-differentiable Lyapunov function. To this end, let \( D^{*}\psi(t) \) denote the upper right-hand derivative of a scalar function $\psi:\mathbb{R}\to\mathbb{R}$, defined by
\[D^{*}\psi(t) := \limsup_{h \to 0^{+}} \frac{\psi(t+h)-\psi(t)}{h}.\]
For a locally Lipschitz continuous function $\psi(\cdot)$, the upper right-hand derivative exists everywhere (see the first paragraph of Section~2 in \cite{naser2017nonsmooth}). In the sequel, we use Rademacher's Theorem, which we review here for the reader's convenience.
\begin{thm}\cite[Theorem~3.2]{evans2025measure}\label{The:Rademacher}~\textbf{Rademacher’s Theorem:}
    A locally Lipschitz continuous mapping $h: \mathbb{R}^n \to \mathbb{R}^m$ is differentiable almost everywhere,  i.e.,  the set of points where $h$ is not differentiable has zero Lebesgue measure.
\end{thm}

The following lemma provides sufficient conditions for fixed-time stability of sets for the system in~\eqref{Eq:Sys}. It plays an important role in the development of the main results herein.

\begin{lem}\label{Lem-2-Kunal's paper}
Let \an{$M\subset\dom(f)$} be the equilibrium set for the system in~\eqref{Eq:Sys}.
Let $V :\mathbb R^m\to \mathbb [0,+\infty)$ be a locally Lipschitz continuous function satisfying the following conditions:
\begin{itemize}
    \item [(1)] If $x\in M$, then $V(x)=0$. Also, there are scalars $\beta_1,\beta_2>0$ such that 
    \[\beta_1{\rm dist}^{\beta_2}(x,M)\leq V(x)\qquad\hbox{for all $x\in\mathbb{R}^m$}. \]
    \item [(2)] For any initial condition \an{$x(0)\in \dom(f)$}, the following holds along the solution $x(\cdot)$ for all $t\geq 0$:
\begin{equation*} 
D^*{V}(x(t)) \le -\left(p\, V(x(t))^{\alpha} + q\, V(x(t))^{\beta}\right)^\theta,
\end{equation*}
where $\alpha,\beta,p, q, \theta > 0$, $\alpha \theta< 1$, and $\beta\theta > 1$.
\end{itemize}
Then, the set $M$ is fixed-time stable for the system in~\eqref{Eq:Sys} and the settling-time $T (x(0))$ is bounded as follows:
\[
T (x(0))\le \frac{1}{p^\theta(1 - \alpha\theta)} + \frac{1}{q^\theta(\beta \theta- 1)} \quad \hbox{for all }x(0)\in{\rm dom}(f).
\]
\end{lem}
\begin{proof}
See Appendix~\ref{Sec:Proof of Lem-2-Kunal's paper}.
\end{proof}

Next, our goal is to relate $D^*V(x(t))$ with the Lie-derivative $\mathcal{L}_fV(x(t))$ for all $t\ge0$.  
The existing results, such as \cite[Theorem~1]{cortes2006finite} and  \cite[Lemma~1]{bacciotti1999stability} 
state that $D^*V(x(t))\in \mathcal{L}_fV(x(t))$ almost everywhere. 
These results apply to differential inclusions and require that $V$ is regular in the sense to be described shortly. Our setting is simpler, and we aim to show that the regularity condition on $V$ is not necessary to relate $D^*V(x(t))$ with the Lie-derivative $\mathcal{L}_fV(x(t))$ for all $t\ge0$. 

The Clarke generalized directional derivative of a locally Lipschitz continuous function $\Psi:\mathbb{R}^m\to \mathbb{R}$ at a point $x\in \mathbb{R}^m$ in the direction $v\in \mathbb{R}^m$ is given by \cite[pg.\ 69-70]{clarke1998nonsmooth}:
\[\Psi^\circ(x;v)=\limsup_{y\to x, s\to0^+}\frac{\Psi(y+sv)-\Psi(y)}{s}.\]
The locally Lipschitz function $\Psi:\mathbb{R}^m\to \mathbb{R}$ is said to be {\it regular at a point} $x\in \mathbb{R}^m$
if the one-sided directional derivative $\Psi'(x;v)$ 
of $\Psi$ at the point $x$ in direction $v$,
\[\Psi'(x;v)=\lim_{s\to0^+}\frac{\Psi(x+sv)-\Psi(x)}{s},\]
exists for all directions $v\in\mathbb{R}^m$ and
\[\Psi^\circ(x;v)=\Psi'(x;v)\qquad\hbox{for all directions } v\in\mathbb{R}^m.\] 
The function $\Psi$ is said to be {\it regular} if it is regular at every $x\in \mathbb{R}^m$ (see \cite[Definition 2.3.4]{clarke1990optimization}).

In the following lemma, for a locally Lipschitz continuous function $V$, we show that $D^*V(x(t))\le\mathcal{L}_fV(x(t))$ for all $t\ge0$ along any solution $x(\cdot)$ of the system in~\eqref{Eq:Sys}, without assuming that the function $V$ is regular. 
\begin{lem}\label{lem: Second-cond}
Let $V:\mathbb{R}^m\to \mathbb{R}$ be locally Lipschitz continuous function, and let $x(\cdot)$ be the solution of~\eqref{Eq:Sys} 
with an arbitrary initial point (see~\eqref{eq:sys-solution}). Then, we have
\[D^*V(x(t))\le \max\mathcal{L}_fV(x(t))\qquad\hbox{for all } t\ge0.\]
\end{lem}
\begin{proof}
See Appendix~\ref{Sec:lem: Second-cond}.
\end{proof}
\section{Unconstrained VI}\label{Sec:unconstrained}
In this section, we consider a VI$(\mathbb{R}^m,F)$ and investigate the stability properties of its solution set ${\rm Sol}(\mathbb{R}^m,F)$ by considering a continuous-time system of the form $\dot x=f(x)$ whose set of equilibria coincides with ${\rm Sol}(\mathbb{R}^m,F)$.

\subsection{Exponential Stability}
We start by showing that the solution set Sol$(\mathbb{R}^m,F)$ can be approached by the trajectories of a continuous-time nonlinear system, when the mapping $F$ is continuously differentiable.
\begin{prop}\label{Prop-sol-set-attract}
    Let $F:\mathbb{R}^m\to \mathbb{R}^m$ be a continuously differentiable mapping, and assume that  ${\rm Sol}(\mathbb{R}^m,F)\ne\emptyset$.
    Let $x_0 \in S\subset \mathbb{R}^m$, with $S$ being a compact and  positively invariant set with respect to dynamic $\dot{x}=-\nabla F(x)^\top F(x)$. Assume that $\lambda_{\min}(\nabla F(x)\nabla F(x)^\top) \neq 0$ for all $x\in S$ where $F(x)\neq 0$. Then, any solution $x:[0,+\infty)\to \mathbb{R}^m$ of $\dot{x}=-\nabla F(x)^\top F(x)$ starting from $x_0$ converges to the largest invariant set $M$ contained in the set ${\rm Sol}(\mathbb{R}^m,F)\cap S$.
\end{prop}
\begin{proof} 
    Let $V(x)=\|F(x)\|^2$ for all $x$. Note that 
    $\dot V(x)=-2F(x)^\top\nabla F(x)\nabla F(x)^\top F(x)\leq 0$ for all $x$. Thus, the result follows by LaSalle's Theorem (\cite[Theorem~4.4]{khalil2002nonlinear}), where we note that the set ${\rm Sol}(\mathbb{R}^m,F)\cap S$ is closed. This follows from
    ${\rm Sol}(\mathbb{R}^m,F)=\{x\in\mathbb{R}^m\mid F(x)=0\}$ being closed ($F$ is continuous) and the assumption that $S$ is compact.
\end{proof}

Next, we will investigate the conditions on the mapping $F$ ensuring that the solution set ${\rm Sol}(\mathbb{R}^m,F)$ is globally exponentially stable, assuming that 
the mapping $F$ is Lipschitz continuous. Specifically, a mapping $F$ is Lipschitz continuous on a set $K\subseteq \mathbb{R}^m$ with a constant $L>0$ when the following relation is satisfied 
\[\|F(x)-F(y)\|\leq L\|x-y\| \quad\hbox{for all $x,y\in K$}.\]
It is Lipschitz continuous if $K=\mathbb{R}^m$ in the preceding relation.

The following result shows that the solution set ${\rm Sol}(\mathbb{R}^m,F)$ is globally exponentially stable
for the nonlinear system of Proposition~\ref{Prop-sol-set-attract}.
In the result, we use
$\lambda_{\min} (A)$ to denote the smallest eigenvalue of a symmetric matrix $A$ and $\lambda_2(A)$ for its second smallest eigenvalue.

\begin{thm}\label{Thm:Design-continuous-dynamic}
    Let the mapping $F : \mathbb{R}^m \to \mathbb{R}^m$ be continuously differentiable and Lipschitz continuous with a constant $L$. Assume that  {\rm Sol}$(\mathbb{R}^m,F)\ne\emptyset$ and 
    there is  $\zeta>0$ such that 
\[\zeta {\rm dist}\left(x,{\rm Sol}(\mathbb{R}^m,F)\right)\le \|F(x)\| \qquad\hbox{for all } x\in\mathbb{R}^m.\]
Consider the dynamical system
$\dot{x} = - \nabla F(x)^\top F(x).$
Assume that there is $\eta>0$ such that, whenever $F(x)\ne 0$, either  
\[\lambda_{\min}(\nabla F(x) \nabla F(x)^\top) \ge \eta,\]
or 
\[\lambda_{\min}(\nabla F(x) \nabla F(x)^\top) =0, \quad \lambda_2(\nabla F(x) \nabla F(x)^\top)\geq \eta,\]
        and $F(x)$ is orthogonal to the eigenspace corresponding to the eigenvalue 0 of $\nabla F(x) \nabla F(x)^\top$.
        Then, the set {\rm Sol}$(\mathbb{R}^m,F)$ is
exponentially stable for the given dynamic.
\end{thm}
\begin{proof}
Note that Sol$(\mathbb{R}^m,F)=\{x\in\mathbb{R}^m\mid F(x)=0\}$, which is a closed set since the mapping $F$ is continuous. Next, we show that Sol$(\mathbb{R}^m,F)$ is the equilibrium set for the given system.
We have  for all $x\in\mathbb{R}^m$,
\[\|\dot x\|^2= F(x)^\top\nabla F(x)\nabla F(x)^\top F(x).\]
Thus,
\[\|\dot x\|^2= 0\qquad\hbox{when }F(x)=0,\]
while by the assumption on the eigenvalues of the matrices $\nabla F(x)\nabla F(x)^\top$, we obtain
\begin{align*}\|\dot x\|^2
&\ge \eta  \|F(x)\|^2>0\qquad\hbox{when }F(x)\ne0.
\end{align*}
Hence, $\dot x=0$ if and only if $F(x)=0$, implying that the set Sol$(\mathbb{R}^m,F)$ is the equilibrium set for the given system.

Consider the Lyapunov function
\[V(x) = \|F(x)\|^2 \quad\hbox{for all } x\in\mathbb{R}^m.\]
Since $\nabla V(x) = 2 \nabla F(x)^{\top}F(x),$
we have that
\[\dot V(x) = \langle \nabla V(x),\dot x\rangle = -2 F(x)^\top \nabla F(x)\nabla F(x)^\top F(x).\]
By the assumption on the eigenvalues of $\nabla F(x)\nabla F(x)^\top$, when $F(x)\ne0$, we have 
\[F(x)^\top \nabla F(x)\nabla F(x)^\top F(x)\ge \eta\|F(x)\|^2,\]
implying that for all $x$ with $F(x)\ne0$,
\[\dot V(x)\le - \eta \|F(x)\|^2.\]
The preceding inequality is trivially satisfied when $F(x)=0$, in which case $\dot V(x)=0$.
Thus, the preceding inequality holds for any $x$.
Using $V(x) = \|F(x)\|^2$, we have that
\[\dot V(x) \le - \eta V(x)\qquad\hbox{for all }x,\]
implying by
 the Gr\"onwall-Bellman inequality that 
\[V(x(t)) \le V(x(0)) e^{-\eta t}\qquad\hbox{for all }t\ge0.\]
Therefore, by using the assumption on the norm $\|F(x)\|$ of the mapping, since $V(x)=\|F(x)\|^2$, we obtain
for all $t>0$,
\begin{align}\label{eq-lowerb}
   \!\!\! \zeta^2{\rm dist}^2\left(x(t),{\rm Sol}(\mathbb{R}^m,F)\right) 
    \leq V(x(t))
    \le V(x(0)) e^{-\eta t}.
    \end{align}

Since $F$ is Lipschitz continuous, for any $x\in\mathbb{R}^m$ and any $y\in {\rm Sol}(\mathbb{R}^m,F)$,  we have $F(y)=0$ and
\[\|F(x)\|\le L\|x-y\|.\]
For a given $x\in\mathbb{R}^m$, let $y_x\in 
{\rm Sol}(\mathbb{R}^m,F)$ be a point such that ${\rm dist}\left(x,{\rm Sol}(\mathbb{R}^m,F)\right)=\|x-y_x\|$. Note that such a point $y_x$ exists since ${\rm Sol}(\mathbb{R}^m,F)$ is closed. Therefore, by letting $y=y_x$ in the preceding inequality, we obtain for all $x\in\mathbb{R}^m$,
\[\|F(x)\|\le L\|x-y_x\|=L\,{\rm dist}\left(x,{\rm Sol}(\mathbb{R}^m,F)\right).\]
Since $V(x)=\|F(x)\|^2$, for any $x(0)$ we have
\begin{align}\label{eq-upperb}
V(x(0))\le L^2{\rm dist}^2\left(x(0),{\rm Sol}(\mathbb{R}^m,F)\right).
 \end{align}
From relations~\eqref{eq-lowerb} and~\eqref{eq-upperb} we obtain that for all $t\ge0$,
\begin{align*}
    {\rm dist}(x(t),{\rm Sol}(\mathbb{R}^m,F))\leq \frac{L}{\zeta}\,{\rm dist}\left(x(0),{\rm Sol}(\mathbb{R}^m,F)\right) e^{-\frac{\eta}{2} t},
\end{align*}
which holds for any initial state $x(0)$.
Hence, 
the solution set Sol$(\mathbb{R}^m,F)$ is
exponentially stable for the stated
system. 
\end{proof}
\subsection{Fixed-time Stability}\label{Ssec: global-fts}
Here, we propose a continuous-time dynamic for which the solution set Sol$(\mathbb{R}^m,F)$ is fixed-time stable.

\begin{thm}\label{Thm:Design-continuous-dynamic-fixed-time-stable}
    Let the mapping $F: \mathbb{R}^m \to \mathbb{R}^m$ be continuously differentiable and Lipschitz continuous with a constant $L>0$. Assume that  Sol$(\mathbb{R}^m,F)$ is a nonempty set and
    there is some $\zeta>0$ such that 
\[\zeta {\rm dist}(x,{\rm Sol}(\mathbb{R}^m,F))\le \|F(x)\|\qquad\hbox{for all } x\in\mathbb{R}^m.\]
Consider the following dynamic: for some \textcolor{black}{$0<\gamma<1$},
\begin{equation}\label{eq:first-fxTS-sys}
    \dot{x} = - \left(\|F(x)\|^{-\gamma}+\|F(x)\| \right)\nabla F(x)^\top F(x),
\end{equation}
and assume that there is $\eta>0$ such that, whenever $F(x)\ne 0$, either  
\[\lambda_{\min}(\nabla F(x) \nabla F(x)^\top) \ge \eta,\]
or 
\[\lambda_{\min}(\nabla F(x) \nabla F(x)^\top) =0, \quad \lambda_2(\nabla F(x) \nabla F(x)^\top)\geq \eta,\]
        and $F(x)$ is orthogonal to the eigenspace corresponding to the eigenvalue 0 of $\nabla F(x) \nabla F(x)^\top$. 
Then, the set Sol$(\mathbb{R}^m,F)$ is FxTS 
for the given system with the following upper bound on the settling time
\[
\sup_{ x(0)\in \mathbb {R}^m} T(x(0))\le \frac{2}{\eta}\, \frac{1+\gamma}{\gamma}.\]
\end{thm}
\begin{proof} 
{\color{black} The modified dynamics in \eqref{eq:first-fxTS-sys} is not locally Lipschitz continuous because of the term with the $-\gamma$ exponent. However, the dynamics is continuous, and the existence and uniqueness of the solutions to~\eqref{eq:first-fxTS-sys} follow from a similar argument to that of\cite[Proposition 1]{garg2022fixed}.} The convergence proof relies on Lemma~\ref{Lem-2-Kunal's paper}. Note that ${\rm Sol}(\mathbb{R}^m,F)=\{x\in\mathbb{R}^m\!\mid\! F(x)=0\}$ and this set is closed since $F$ is continuous mapping.
Next, we show that ${\rm Sol}(\mathbb{R}^m,F)$ is
the equilibrium set of the system in~\eqref{eq:first-fxTS-sys}.
For the system in~\eqref{eq:first-fxTS-sys}, we have for all $x\in\mathbb{R}^m$,
\[\|\dot x\|^2
=\left(\|F(x)\|^{-\gamma}+\|F(x)\| \right)^2 F(x)^\top\nabla F(x)\nabla F(x)^\top F(x).\]
Therefore,
\[\|\dot x\|^2= 0\qquad\hbox{if }F(x)=0,\]
while by the assumption on the eigenvalues of the matrices $\nabla F(x)\nabla F(x)^\top$, we have
\[\|\dot x\|^2
\ge \eta\! \left(\|F(x)\|^{-\gamma}\!+\!\|F(x)\| \right)^2 \|F(x)\|^2>\!0
\quad\hbox{if }F(x)\ne0.\]
Hence, $\dot x=0$ if and only if $F(x)=0$, implying that the set Sol$(\mathbb{R}^m,F)$ is  the equilibrium set for \eqref{eq:first-fxTS-sys}. 

Now, we consider the following Lyapunov function:
\[V(x) = \|F(x)\|^2 \qquad\hbox{for all }x\in\mathbb{R}^m.\]
By continuous differentiability of the mapping $F$, the function $V$ is also continuously differentiable, hence, it is locally Lipschitz continuous. 
Since
\[\{x\in\mathbb{R}^m\!\mid\! V(x)\!=0\}=\{x\in\mathbb{R}^m\!\mid\! F(x)=0\}={\rm Sol}(\mathbb{R}^m,F),\]
and \[\zeta^2 {\rm dist}^2(x,{\rm Sol}(\mathbb{R}^m,F))\le \|F(x)\|^2\qquad\hbox{for all } x\in\mathbb{R}^m,\] 
condition~(1) of Lemma~\ref{Lem-2-Kunal's paper} is satisfied.

Next, we show that condition (2) of Lemma~\ref{Lem-2-Kunal's paper} is satisfied.
First, we note that the function $t\mapsto V(x(t))$, $t\ge0$, is differentiable,
so $D^*(V(x(t))=\langle\nabla V(x(t)),\dot x(t)\rangle$ for all $t\ge0$.
Since
$\nabla V(x) = 2 \nabla F(x)^{\top}F(x)$ for all $x\in\mathbb{R}^m$,
by using the dynamic system equation~\eqref{eq:first-fxTS-sys}, we obtain
\begin{align*}
   &\dot V(x) = \langle \nabla  V(x), \dot x\rangle
   =
-2 \left(\|F(x)\|^{-\gamma}+\|F(x)\| \right)F(x)^\top \nabla F(x)\nabla F(x)^\top F(x).
\end{align*}
Using the assumption on the eigenvalues of $\nabla F(x)\nabla F(x)^\top$, when $F(x)\ne0$, we have
\[F(x)^\top \nabla F(x)\nabla F(x)^\top F(x)\ge \eta\|F(x)\|^2,\] so 
we obtain for all $x$ with $F(x)\ne0$,
\begin{eqnarray}\label{eq-th2-aux-1}
\dot V(x) \le - \eta (\|F(x)\|^{2-\gamma}+\|F(x)\|^3).\nonumber
\end{eqnarray}
Since $\dot V(x)= 0$ when $F(x)=0$,  the preceding relation is trivially satisfied (as equality). Thus, the preceding inequality holds for all $x\in\mathbb{R}^m$, and using $V(x) = \|F(x)\|^2$, we have
\[\dot V(x) \le - \eta \left(V(x)^{1-\frac{\gamma}{2}}+V(x)^{\frac{3}{2}}\right) \qquad\hbox{for all }x\in\mathbb{R}^m.\]
Thus, the condition (2) of Lemma~\ref{Lem-2-Kunal's paper} is satisfied
with $p=q=\eta$, $\theta=1$, $\alpha=1-\gamma/2\in (0, 1)$ given the bounds on $\gamma$, and $\beta=3/2$.
Since all the conditions of  Lemma~\ref{Lem-2-Kunal's paper} are satisfied, by the lemma, the solution set ${\rm Sol}(\mathbb{R}^m, F)$ is
fixed-time stable with the following upper bound on the settling time
\[\sup_{ x(0)\in \mathbb {R}^m} T(x(0)) \le \frac{2}{\eta \gamma} + \frac{2}{\eta},\]
which yields the stated bound.
\end{proof}

\section{Constrained VI}\label{sec:contrained}
Here, we consider a non-monotone VI$(K,F)$ and propose continuous-time dynamics for which the solution set 
Sol$(K,F)$ is globally exponentially and fixed-time stable.

\subsection{Exponential Stability}
We focus on a VI$(K,F)$ with a (nonempty) closed and convex set $K\subset\mathbb{R}^m$ and a mapping $F:K\to\mathbb{R}^m$. 
The discrete-time Korpelevich method is given by \sa{\cite[Eq.~(2)]{korpelevich1976extragradient}}
\begin{align*}
   y_k &= \Pi_K[x_k - \alpha F(x_k)],\\
    x_{k+1} &= \Pi_K[x_k - \alpha F(y_k)],
    \end{align*}  
initialized with $x_0,y_0\in K$, where $\alpha>0$ is a stepsize. 
We consider a continuous-time variant of this method in the following form\footnote{Given Lipschitz continuity of $F$, the dynamic is Lipschitz by Lipschitzness of the projection operator with the Lipschitz constant $L=1$. The differentiability of $F$ similarly implies local Lipschitzness of the dynamic.}:
\begin{subequations}\label{eq: ct Korpelevich}
    \begin{align}
        \dot y & = \Pi_K[x-\alpha F(x)]-y,\\
        \dot x & = \Pi_K[x-\alpha F(y)]-x,
    \end{align}
\end{subequations}
with $x(0),y(0)\in K$, where $\alpha>0$ is some constant.
The iterates of the discrete-time Korpelevich method are always in the set $K$ where the mapping is defined.  For the continuous-time Korpelevich method~\eqref{eq: ct Korpelevich}, it is not obvious if the trajectories $x(\cdot)$ and $y(\cdot)$ remain in the set~$K$.
In the following lemma, for a closed convex set $K$, we show that this is indeed the case provided that the trajectories are initialized in the set $K$. 
\begin{lem}\label{lem:traj-in K} 
Let $K\subset\mathbb{R}^m$ be a closed convex set. Then,
    the set $K\times K$ is positively invariant with respect to the continuous-time Korpelevich system in~\eqref{eq: ct Korpelevich}.
\end{lem}
\begin{proof} We write the system in~\eqref{eq: ct Korpelevich} in an alternative form.
Define $z_1=y$, $z_2=x$, and $z=[z_1^\top,z_2^\top]^\top\in K\times K$. Then, we define the mapping $H(z)=[H_1(z)^\top,H_2(z)^\top]^\top:K\times K\to\mathbb{R}^m\times \mathbb{R}^m$ with components $H_1(z)=\Pi_K[z_2-\alpha F(z_2)]$ and $H_2(z)=\Pi_K[z_2-\alpha F(z_1)]$. Then, the system in~\eqref{eq: ct Korpelevich} has the following form:
\begin{align}\label{eq: ct Korpelevich-alt}
        \dot z & = H(z)-z.
\end{align}
Let $t\ge0$ be arbitrary.
    Multiplying \eqref{eq: ct Korpelevich-alt} by $e^t$, we obtain
    \[(\dot z(t)+z(t))e^t=H(z(t))e^t,\]
    implying  that
    \[\int_0^td(z(s)e^s)=\int_0^tH(z(s))e^sds.\]
    Therefore,
    \[z(t)e^t=z(0)+\int_0^tH(z(s))e^sds,\]
    and consequently,
    \[z(t)=\frac{z(0)}{e^t}+\frac{e^t-1}{e^t}\int_0^tH(z(s))\frac{e^s}{e^t-1}ds.\]
    We view the integral $\int_0^tH(z(s))\frac{e^s}{e^t-1}ds$ as the expected value of $H(z(s))$ for $s\in [0,t]$ with respect to the probability distribution function $g(s)=\frac{e^s}{e^t-1}$ for $s\in [0,t]$. Since $K$ is convex, so is the set $K\times K$. Therefore, the integral value $\int_0^tH(z(s))\frac{e^s}{e^t-1}ds$ also belongs to
    $K\times K$. Given $z(0)\in K\times K$, the point $z(t)$ is a convex combination of two vectors in $K\times K$, so it belongs to $K\times K$ since $K\times K$ is convex. Thus, the set $K\times K$ is positively invariant for the Korpelevich system in~\eqref{eq: ct Korpelevich}.
\end{proof}

To proceed with the stability properties of the Korpelevich system in~\eqref{eq: ct Korpelevich},
we use an alternative characterization of the solution set of  VI$(K,F)$, by using a {\it scaled} natural mapping $F_{K,\tau}^{\rm nat}:K\to \mathbb{R}^m$,  with  $\tau>0$, defined by 
\begin{equation}\label{eq-def-fnat}
    F_{K,\tau}^{\rm nat}(x)=x-\Pi_K\left[x-\tau F(x)\right]
\qquad\hbox{for all $x\in K$}.
\end{equation}
The following result relates the solutions of the VI$(K,F)$ problem with the zeros of the scaled natural mapping $F^{\rm nat}_{K,\tau}$.
\begin{lem}\label{lem-auxiliary-natural map solution}
    Let $K\subseteq\mathbb R ^m$ be a closed convex set, and let
    $F:K\to \mathbb{R}^m$ be a mapping. Then, we have for any $\tau>0$,
    \begin{equation*}
        x^*\in {\rm Sol}(K,F)\quad \iff \quad F_{K,\tau}^{\rm nat}(x^*)=0.
    \end{equation*}
\end{lem}
\begin{proof} See \cite[Proposition~1.5.8]{facchinei2003finite}, and the discussion therein on page 85, where one can use $A=\tau^{-1}\mathbb{I}$.
\end{proof}

Next, we relate the equilibrium points of the Korpelevich system in~\eqref{eq: ct Korpelevich}  with the solution set of VI$(K,F)$.
\begin{lem}\label{lem-auxiliary-equilibria}
Let $K\subseteq \mathbb{R}^m$ be a closed convex set and $F:K\to \mathbb{R}^m$ be a Lipschitz continuous mapping with a Lipschitz constant $L>0$. Then, for any initial points $x(0),y(0)\in K$ and $0<\alpha < \frac{1}{L}$, the equilibrium set of the Korpelevich system in~\eqref{eq: ct Korpelevich}  coincides with the set $\{(x^*,x^*)\mid x^*\in{\rm Sol}(K,F)\}$.
\end{lem}
\begin{proof}
A point $(\bar x,\bar y)\in K\times K$ is an equilibrium of \eqref{eq: ct Korpelevich} if and only if the following relations holds:
\begin{subequations}\label{eq: ct Korpelevich-equilibrium}
    \begin{align}
       \bar y & =\Pi_K[\bar x-\alpha F(\bar x)],\\
       \bar x & =\Pi_K[\bar x-\alpha F(\bar y)].
    \end{align}
\end{subequations}

Suppose that $(\bar x,\bar y)\in K\times K$ is an equilibrium point of \eqref{eq: ct Korpelevich}. By~\eqref{eq: ct Korpelevich-equilibrium}, using the non-expansiveness property of the projection and Lipschitz continuity of the mapping $F$, we obtain 
\[\|\bar x-\bar y\|\leq \alpha\|F(\bar x)-F(\bar y)\|\le \alpha L\|\bar x-\bar y\|.\] 
Since $\alpha L<1$ it follows that $\bar x= \bar y$, which when used in \eqref{eq: ct Korpelevich-equilibrium}(b) yields $F_{K,\alpha}^{\rm nat}(\bar x)=0$. By Lemma~\ref{lem-auxiliary-natural map solution}, we have $\bar x\in {\rm Sol}(K,F)$. 

 Assume now that $\bar x\in {\rm Sol}(K,F)$. By Lemma \ref{lem-auxiliary-natural map solution} we have $\bar x=\Pi_K[\bar x-\alpha F(\bar x)]$, and with $\bar y=\bar x$,
we can obtain \eqref{eq: ct Korpelevich-equilibrium}.
\end{proof}

To investigate the stability properties of the Korpelevich system in~\eqref{eq: ct Korpelevich}, 
we will use a relaxation of $\eta$-quasi strongly monotone assumption for a VI($K,F$)~\cite[Assumption~2.1]{zhang2023communication}.
{\color{black}    \begin{assum}\label{Assum-nu-quasi-monotone}
    Let $K\subseteq\mathbb{R}^m$, $F:K\to\mathbb{R}^m$, and Sol$(K,F)$ be a nonempty. For every $x\in K$, let 
     $S(x)$ be the set of points in {\rm cl(Sol($K,F$))} that attain the distance from $x$ to the closure of the set {\rm Sol($K,F$)}.
     Assume that there is some $\eta>0$ such that, for all $x\in K$, there is $v^*_x\in S(x)$ satisfying the 
     following relation
    \[\langle F(x)-F(v^*_x),x-v^*_x\rangle\!\!\geq \eta\,\dist^2(x,\!{\rm Sol}(K,F)).\]
\end{assum}
}

\an{The set $S(x)$ of the projection points of $x$ on the closure of Sol$(K,F)$ is nonempty and compact for all $x$. In all the results that follow, we deal with a closed set $K$ and a continuous mapping 
$F:K\to\mathbb{R}^m$, so the solution set Sol$(K,F)$ is always closed while possibly empty. Thus, in Assumption~\ref{Assum-nu-quasi-monotone}, the closure of Sol$(K,F)$ will not be needed.}

\an{When VI$(K,F)$ has a unique solution $x^*$, then $S(x)=\{x^*\}$ for all $x$, and the condition of Assumption~\ref{Assum-nu-quasi-monotone} reduces to
\[\langle F(x)-F(x^*),x-x^*\rangle\geq \eta\,\|x-x^*\|^2\quad\hbox{for all } x\in K.\]
}
We represent the system in~\eqref{eq: ct Korpelevich} in the following form:
\begin{subequations}\label{eq: ct Korpelevich-v3vector}
    \begin{align}
        &\dot y  = H_1(x,y),  \qquad \dot x  = H_2(x,y),\\
        & H_1(x,y)=\Pi_K[x-\alpha  F(x)]-y,\\
        & H_2(x,y) = \Pi_K[x-\alpha  F(y)]-x,
    \end{align}
\end{subequations}
with $x(0)$,$y(0)\in K$. As VI$(K,F)$ may have multiple solutions,
we  consider a non-smooth Lyapunov function involving the squared distance function ${\rm dist}^2(z,U)$ from a vector $z\in \mathbb{R}^m$ to a closed set $U\subset\mathbb{R}^m$. 
The function ${\rm dist}^2(z,U)$ is locally Lipschitz, and the Clarke subdifferential $\partial \,{\rm dist}^2(z,U)$ is nonempty, convex, and compact set for all $z\in \mathbb{R}^m$ \cite[Proposition~2.1.2]{clarke1990optimization}. In particular, we have the following result.

\begin{lem}\label{lem-sudiff-dist2}
Let $U\subset\mathbb{R}^m$ be a nonempty closed set. Then, the subdifferential $\partial \,{\rm dist}^2(\cdot,U)$ is given by
\[
\partial{\rm dist}^2(z,U)\!=\!\!\left\{\begin{array}{ll}
    \!\! \! 2\,\left(z-{\rm conv}(U_z)\right) &\!\! \hbox{if $z\notin U$},\cr
\!\! \! \an{\{0\}} & \!\!\hbox{if $z\in U$},
\end{array}\right.
\]
where $U_z$ is the compact subset of $U$ consisting of the points that attain the minimum distance from $z$ to the set $U$.
\end{lem}
\begin{proof}
The proof is in Appendix~\ref{ap-proof-subdiff}.
\end{proof}

Next, we discuss the Lie derivative.
Let $h:\mathbb{R}^m\to \mathbb{R}^m$, and consider the differential equation $\dot z=h(z)$.
The Lie derivative of a \sa{locally Lipschitz} function $V:\mathbb{R}^m\to \mathbb{R}$ with respect to the mapping $h$ is given by 
\[\mathcal{L} _h V(z):=\{\langle\xi, \dot z\rangle\mid \xi\in \partial V(z)\}.\] 
\an{The Lie derivative $\mathcal{L}_h V(z)$ of $V(z)$ with respect to $h$ is a compact set.  To see this, note that 
$\partial V(z)$ is nonempty compact convex set by \cite[Proposition~2.1.2]{clarke1990optimization}.
Thus, by using
the Cauchy-Schwarz inequality, we can see that the set $\mathcal{L} _h V(z)$ is bounded since $\partial V(z)$ is bounded for any $z \in \mathbb{R}^m$. By the compactness of the set $\partial V(z)$ and the continuity and linearity of the inner product, the set $\mathcal{L} _h V(z)$ is compact. Moreover, we have the following result that we will apply to $\max_{\xi\in\conv(\partial V(z))} \mathcal{L} _h V(z)$.}

\sa{
\begin{lem}\cite[Theorem~9.4.6]{narici2010topological}\footnote{
    The theorem is for sets in a locally convex Hausdorff space. The set $\mathbb{R}^m$ with the standard Euclidean norm is a locally convex Hausdorff space.}
\label{lemma-Krein-Milman}
    Let $C\subset\mathbb{R}^m$ be a compact convex set and let $B\subset C$. Then, $C={\rm cl}(\conv(B))$ if and only if the extreme
    points\footnote{A point $x$ is an extreme point of a convex set $X$ if $x$ cannot be obtained as a convex combination of two other distinct points in the set $X$.}  of $C$ are contained in the set ${\rm cl(B)}$.
\end{lem}
When a set $B$ is compact, its convex hull is also compact. As a direct consequence of Lemma~\ref{lemma-Krein-Milman}, for a compact set $B$, we have that 
$C=\conv(B)$ if and only if the extreme
    points of $C$ are contained in the set $B$.
A linear function constrained to a compact convex set $C=\conv(B)$ always attains its maximal (or minimal) value at an extreme point of the set $C$, which by the preceding result implies that for any $\mathbf{c}\in\mathbb{R}^m$,
\begin{equation}\label{eq-milman}
\max_{v\in \conv(B)}\langle \mathbf{c},v\rangle =\max_{v\in B}\langle \mathbf{c},v\rangle.
\end{equation}
}

For the system in~\eqref{eq: ct Korpelevich}, we have $z=[y^\top,x^\top]^\top$ and $\dot z=H(z)$,
with $H(z)=[H_1(z)^\top,H_2(z)^\top]^\top$. 
Due to the Cartesian structure of the dynamic, we have
\[\partial V(z)=\{(\xi_y,\xi_x)\mid \xi_y\in \partial_y V(z), \    \xi_x\in \partial_x V(z)\},\]
\[\mathcal{L}_H V(z)=\{\langle\xi_x,\dot x\rangle+\langle\xi_y,\dot y\rangle \mid \xi_x\in \partial_x V(z), \ \xi_y\in \partial_y V(z)\}.\]

The solution set ${\rm Sol}(K,F)$ is exponentially stable for the Korpelevich system in~\eqref{eq: ct Korpelevich}, as seen in the following result.
\begin{thm} \label{thm-Korpelevich-analysis}
    Let $K\subseteq \mathbb{R}^m$ be a closed convex set and 
    $F:K\to\mathbb{R}^m$ be a Lipschitz continuous mapping with a constant $L>0$. Let Assumption~\ref{Assum-nu-quasi-monotone} hold. Also, let $0<\alpha <\min\{\frac{\eta}{L^2},\frac{1}{2L}\}$. Then, the set $\{(x^*,x^*)\mid x^*\in{\rm Sol}(K,F)\}$ is exponentially stable for the Korpelevich system in~\eqref{eq: ct Korpelevich}, initialized at any  $x(0),y(0)\in K$.
\end{thm}
\begin{proof}
The solution set {\rm Sol}$(K,F)$ is nonempty by Assumption~\ref{Assum-nu-quasi-monotone},
and is closed since the set $K$ is closed and the mapping $F$ is continuous on $K$.
Since $\alpha L<1$, by Lemma~\ref{lem-auxiliary-equilibria}
 the
equilibrium set of~\eqref{eq: ct Korpelevich} coincides with the set {\rm Sol}$(K,F)$. 

Define a candidate Lyapunov function $V:\mathbb R^m\times\mathbb R^m\to [0,+\infty)$ by: for all $x,y\in \mathbb{R}^m$,
\begin{align*}
    V(x, y) = \frac{1}{2}\left({\rm dist}^2(x,{\rm Sol}(K,F)) + \|x - y\|^2\right).
\end{align*}
Let $S(x)$ be the set of projection points of $x$ on  
${\rm Sol}(K,F)$, 
\[S(x)=\{v\in {\rm Sol}(K,F)\mid \|x-v\|={\rm dist}\left(x,{\rm Sol}(K,F)\right)\}.\]
We have 
\[\nabla_y V(x,y)=y-x,\]
while by Lemma~\ref{lem-sudiff-dist2} (since $V$ is locally Lipschitz) we have
\begin{align*}
 \partial_x V(x,y) &= x-y +\!\left\{\begin{array}{ll}
    \!\! \! x-{\rm conv}(S(x)) &  \hbox{if $x\notin {\rm Sol}(K,F)$},\cr
\!\! \! \an{\{0\}}& \hbox{otherwise}.
\end{array}\right.   
\end{align*}
For the Lie derivative of $V$ along the trajectories of \eqref{eq: ct Korpelevich}, using the decomposable structure of $\partial V$, we have
\[\mathcal{L}_H V(x, y)=\{\langle u,\dot x\rangle +\langle y-x,\dot y\rangle \mid u\in\partial_xV(x,y)\}.
\]
Using the expression for
$\partial_x V(x,y)$ and taking the maximum, we obtain 
\begin{align*}
  &  \max\mathcal{L}_H V(x, y)=\langle x-y,\dot x-\dot y\rangle\ + \left\{\!\!\!
\begin{array}{ll}
    \max_{v\in {\rm conv}(S(x))}\langle x-v,\dot x\rangle & \!\!\! \hbox{if $x\notin {\rm Sol}(K,F)$},\cr
    \an{\{0\}}
    & \!\!\!\hbox{otherwise}.
\end{array}\right.   
\end{align*}
\an{
Since the set $S(x)$ is compact, by Lemma~\ref{lemma-Krein-Milman}, we have that  relation~\eqref{eq-milman} holds with $B=S(x)$, i.e.,
\[\max_{v\in \conv(S(x))}\langle x-v,\dot x\rangle =\max_{v\in S(x)}\langle x-v,\dot x\rangle,\]
}
so we obtain 
\begin{align*}
  &  \max\mathcal{L}_H V(x, y)=\langle x-y,\dot x-\dot y\rangle \ + \left\{\!\!\!
\begin{array}{ll}
    \max\{\langle x-v,\dot x\rangle|\, v\in S(x)\} & \!\!\! \hbox{if $x\notin {\rm Sol}(K,F)$},\cr
    \an{\{0\}}
    & \!\!\!\hbox{otherwise}.
\end{array}\right.   
\end{align*}
Now, we consider two cases separately, namely, the case when $x\notin {\rm Sol}(K,F)$, and the case when $x\in {\rm Sol}(K,F)$ but $y\ne x$.

\noindent
{\it Case $x\notin {\rm Sol}(K,F)$.}
For the Lie derivative, we have
\begin{align}\label{eq-liederiv-H}
  &  \max\mathcal{L}_H V(x, y)=\langle x-y,\dot x-\dot y\rangle+
    \max_{v\in S(x)} \langle x-v,\dot x\rangle.
    \end{align}
Using the system dynamic~\eqref{eq: ct Korpelevich}, we obtain
\begin{align*}
    \langle x-y,\dot x-\dot y\rangle 
    &= -\|x-y\|^2 \ +\langle x-y, \Pi_K[x-\alpha F(y)]-\Pi_K[x-\alpha F(x)]\rangle\cr&\le 
    -\|x-y\|^2 +\|x-y\|\,\|\Pi_K[x-\alpha F(y)]-\Pi_K[x-\alpha F(x)]\|.
\end{align*}
By using the non-expansiveness property of the projection mapping $\Pi_K$ and the Lipschitz continuity of $F$, we obtain
\begin{align}\label{eq-liederiv-p1-H}
    \langle x-y,\dot x-\dot y\rangle \le -(1-\alpha L)\|x-y\|^2.
\end{align}

Now, we consider the set on the right hand side in~\eqref{eq-liederiv-H}. For an arbitrary $v\in S(x)$, we have
\begin{align*}
    \langle x-v,\dot x\rangle &=
    \langle x-v, \Pi_K[x-\alpha F(y)]-x\rangle
\end{align*}
Since $S(x)\subset {\rm Sol}(K,F)$, the point $v$ is a solution to VI($K,F)$. Thus, by Lemma~\ref{lem-auxiliary-natural map solution},   it satisfies $v-\Pi_K[v-\alpha F(v)]=0$, and by using this equality in the preceding relation, we obtain
\begin{align*}
    \langle x-v,\dot x\rangle &=\langle x-v, v-\Pi_K[v-\alpha F(v)] \rangle +\langle x-v,\Pi_K[x-\alpha F(y)]-x\rangle\cr
    & = \langle x-v, \Pi_K[x-\alpha F(y)] -\Pi_K[v-\alpha F(v)]\rangle -\|x-v\|^2\\
    &\le -\|x-v\|^2 +\|x-v\|\, \|\Pi_K[x-\alpha F(y)] -\Pi_K[v-\alpha F(v)]\|\cr 
    &\le -\frac{1}{2}\|x-v\|^2 + \frac{1}{2}\|\Pi_K[x-\alpha F(y)] -\Pi_K[v-\alpha F(v)]\|^2,
\end{align*}    
where the last inequality follows from $ab\le (a^2+b^2)/2$ valid for any scalars $a,b$.
To bound $\|\Pi_K[x-\alpha F(y)] -\Pi_K[v-\alpha F(v)]\|^2$, we use  Lemma~\ref{Lem-aux-1}(a) with $\tau=\alpha$, and obtain
\begin{align*}
    \langle x-v,\dot x\rangle 
    & \le \frac{c_\alpha^2 -1}{2}\|x-v\|^2 =-\left(\alpha\eta-\frac{\alpha^2 L^2}{2}\right) 
    \|x-v\|^2.
\end{align*}   
As $\|x-v\|=\dist(x,{\rm Sol}(K,F))$, we obtain for any $v\in S(x)$,
\begin{equation}\label{eq-liederiv-p2-H}
\langle x-v,\dot x\rangle 
    \le -\left(\alpha\eta-\frac{\alpha^2 L^2}{2}\right) \dist^2(x,{\rm Sol}(K,F)).
    \end{equation}
Using relations~\eqref{eq-liederiv-p1-H} and~\eqref{eq-liederiv-p2-H}
in equation~\eqref{eq-liederiv-H}, we find that
\begin{align*}
&\max\mathcal{L}_H V(x, y) \le 
-  (1-\alpha L)\|x-y\|^2 -\left(\alpha\eta-\frac{\alpha^2 L^2}{2}\right)\dist^2(x,{\rm Sol}(K,F)).
  \end{align*}
We have
$d=\min\left\{\alpha\eta-\alpha^2L^2,\frac{1}{2}-\alpha L\right\}>0$ 
for $0<\alpha <\min\{\frac{\eta}{L^2},\frac{1}{2L}\}$. 
Hence, in terms of the Lyapunov function $V$, 
\begin{align}\label{eq-liederiv3-H} 
    \max \mathcal{L}_H V(x, y)\leq -2d V(x,y).
\end{align}

\noindent
{\it Case $x\in {\rm Sol}(K,F)$ and $y\ne x$.} The Lie derivative is
\[\mathcal{L}_HV(x,y) =\langle x-y,\dot x-\dot y\rangle. \]
Using relation~\eqref{eq-liederiv-p1-H}, and noting that $V(x,y)=\|x-y\|^2/2$ when $x\in{\rm Sol}(K,F)$,
we find that relation \eqref{eq-liederiv3-H} holds. 

In both of these two cases, since the function $V$ is locally Lipschitz and the solution of \eqref{eq: ct Korpelevich} is unique for any initial condition, using Lemma~\ref{lem: Second-cond}, we obtain that  $D^*V(x(t),y(t))\leq \max \mathcal{L}_HV(x(t), y(t))$. Hence, for almost all $t\geq 0$,
\[
   \dot V(x(t),y(t))= D^*V(x(t),y(t))\leq -2d V(x(t),y(t)).
\]
By the Gr\"onwall-Bellman inequality~\cite[Theorem~1.3]{wang2015generalization}, 
\[   V(x(t),y(t))\leq V(x(0),y(0))e^{-2dt}\quad \hbox{for all $t\ge0$}.\]
Moreover, for $x^*$ in the projection set of $x$ on ${\rm Sol}(K,F)$, we have 
\begin{align*}
     {\rm dist}^2((x,y),\{(z,z)\mid z\in {\rm Sol}(K,F) \})& = \|x-x^*\|^2+\|y-x^*\|^2\nonumber\\
     &\leq \|x-x^*\|^2+(\|y-x\|^2+\|x-x^*\|^2)\nonumber\\
     &\leq 3(\|x-x^*\|^2+\|y-x\|^2)\nonumber\\
     &\leq 6V(x,y).
\end{align*}
Additionally, for any $x^*\in {\rm Sol}(K,F)$,
\begin{align*}
     2V(x, y)& ={\rm dist}^2(x,{\rm Sol}(K,F))+\|x-y\|^2\cr
     &={\rm dist}^2(x,{\rm Sol}(K,F))+\|x-x^*+x^*-y\|^2 \nonumber\\
     &\leq {\rm dist}^2(x,{\rm Sol}(K,F))+2\|x-x^*\|^2+2\|x^*-y\|^2.
\end{align*}
Taking infimum over $x^*\in {\rm Sol}(K,F)$ yields
\begin{align}\label{eq:uperbound v by dist}
2V(x, y)\leq &\, {\rm dist}^2(x,{\rm Sol}(K,F))+2{\rm dist}^2((x,y),\{(x^*,x^*)\mid x^*\in {\rm Sol}(K,F) \}), \nonumber\\
     \leq & 3\, {\rm dist}^2((x,y),\{(x^*,x^*)\mid x^*\in {\rm Sol}(K,F) \}).
\end{align}
Thus, the solution set is globally exponentially stable.
\end{proof}

\subsection{Fixed-time Stability}
We next discuss an accelerated variant of the continuous-time  Korpelevich method in~\eqref{eq: ct Korpelevich}, given by
\begin{subequations}\label{eq: ct Korpelevich-v2}
    \begin{align}
        \dot y & = \rho_{\lambda,\zeta}(x,y)(\Pi_K[x-\alpha  F(x)]-y),\\
        \dot x & = \rho_{\lambda,\zeta}(x,y)(\Pi_K[x-\alpha  F(y)]-x),
    \end{align}
\end{subequations}
where $\alpha>0$,  $x(0),y(0)\in K$, and $\rho_{\lambda,\zeta}:K\times K\to[0,+\infty)$ is a function parameterized by scalars $\lambda,\zeta>0$ and such that 
$\rho_{\lambda,\zeta}(x,y)=0$ only for $x\in{\rm Sol}(K,F)$ and $y=x$. First, we show that the set $K\times K$ is positively invariant for the system in~\eqref{eq: ct Korpelevich-v2}.
\begin{lem}\label{lem:taj-in K-sys-2}
    Let $K\subseteq \mathbb{R}^m$ be a closed convex set {\color{black}and $\rho_{\lambda,\zeta}$ be such that the dynamics in \eqref{eq: ct Korpelevich-v2} is continuous over $K\times K$}. Then, the set $K\times K$ is positively invariant with respect to the system~\eqref{eq: ct Korpelevich-v2}. Moreover, if $F:K\to\mathbb{R}^m$ is Lipschitz continuous over the set $K$
    with a constant $L>0$, then for $\alpha>0$ with $\alpha L<1$, the equilibrium points of the system in~\eqref{eq: ct Korpelevich-v2} coincide with the set $\{(x^*,x^*)\mid x^*\in {\rm Sol}(K,F)\}$.    
\end{lem}
\begin{proof}
By Lemma~\ref{lem:traj-in K}, the set $K\times K$ is positively invariant for the system in~\eqref{eq: ct Korpelevich-alt}. 
By~\cite[Theorem~3.1]{BLANCHINI19991747} (or~\cite[Proposition~4.6.2]{bno2003convex}), the vector $H(z)-z$, with $z=(y,x)$, belongs to the tangent cone of the set $K\times K$ at $z$, where
the tangent cone of a convex set $C$ at a point $z\in C$ is defined by (see~\cite[Theorem~6.9]{rockafellar1998variational}) 
\[ T_C(z)={\rm cl}\{w\mid \exists \tau >0 \ {\rm with} \ z+\tau w\in C\},\] where
${\rm cl}(\cdot)$ denotes the closure of a set.
By the property of cones, for any \an{$\rho_{\lambda,\zeta}(z)\ge0$}, the vector $\rho_{\lambda,\zeta}(z)(H(z)-z)$ remains in the tangent cone of the set $K\times K$ at $z$. By~\cite[Theorem~3.1]{BLANCHINI19991747}, the set $K\times K$ is positively invariant with respect to the system $\dot z=\rho_{\lambda,\zeta}(z)(H(z)-z)$. 

By the definition of $\rho_{\lambda,\zeta}$, we have that the point $(\bar x,\bar y)\in K\times K$, with $\bar x\in {\rm Sol}(K,F)$, $\bar y=\bar x$,
is an equilibrium point of the system in~\eqref{eq: ct Korpelevich-v2}.
When $\rho_{\lambda,\zeta}(x,y)>0$, an equilibrium point satisfies 
    \begin{align*}
       y  =\Pi_K[ x-\alpha F( x)],\quad
       x  =\Pi_K[ x-\alpha F(y)].
    \end{align*}
Using a similar analysis to that of the proof of Lemma~\ref{lem-auxiliary-equilibria}, we will conclude that $x\in{\rm Sol}(K,F)$ and $y=x$, which contradicts the fact that $\rho_{\lambda,\zeta}(x,y)>0$.
\end{proof}

We are interested in selecting the function $\rho_{\lambda,\zeta}$ so that the set $\{(x^*,x^*)\mid x^*\in {\rm Sol}(K,F)\}$ is fixed-time stable for the system in~\eqref{eq: ct Korpelevich-v2}.
Toward that goal, we define
\begin{equation}\label{eq: def_S_lambda}
    S_\lambda(x,y)=\|F_{K,\lambda}^{\rm nat}(x)\|^2+\|x-y\|^2\quad\hbox{for all }x,y\in K,
\end{equation}
and, for some $0<\zeta<\frac{1}{2}$, define $\rho_{\lambda,\zeta}(x,y)$ as follows:
\begin{equation}\label{eq: def_rho_lambda}
    \rho_{\lambda,\zeta}\!(x,y)\!\!=\!\!\left\{\!
    \begin{array}{ll}
         \!\!\!S_\lambda(x,y)^{\frac{1}{2}}\!+\!S_\lambda(x,y)^{\zeta-\frac{1}{2}} \!\!\!&\hbox{if $S_\lambda(x,y)\!>\!0$},\\
         \!\!0\!\!\! & \hbox{if $S_\lambda(x,y)\!=\!0$}.
    \end{array}\right.
\end{equation}
Let us also define for all $x,y\in \mathbb{R}^m$,
\begin{equation}\label{eq: def_V}
    V(x, y) = \frac{1}{2}\big({\rm dist}^2(x,{\rm Sol}(K,F)) + \|x - y\|^2\big).
\end{equation}

\begin{lem}\label{Lem-aux-1}
Let $K\subseteq \mathbb{R}^m$ be a closed convex set and $F: K\to \mathbb{R}^m$ be a Lipschitz continuous mapping with a Lipschitz constant $L>0$. Let Assumption~\ref{Assum-nu-quasi-monotone} hold with some $\eta>0$. Let  $\tau \in (0,\frac{2\eta}{L^2})$ and define 
$c_\tau=\sqrt{1-2\tau \eta +\tau^2 L^2}.$ Then,  $c_\tau\in [0,1)$ and we have:\\
\noindent
(a)
{\color{black} For all $x\in K$,
\[\|\Pi_K[x-\tau F(x)]\!-\!\Pi_K[v^*_x-\tau F(v^*_x)]\|\leq c_\tau\,
\dist(x,{\rm Sol}(K,F)),\]
where $v_x^* \in {\rm Sol}(K,F)$ satisfies Assumption~\ref{Assum-nu-quasi-monotone}.}\\
\noindent
(b) For $\lambda \in (0,\frac{2\eta}{L^2})$ and 
for  all $(x,y)\in \mathbb{R}^m\times \mathbb{R}^m\setminus\{(x^*,x^*)\mid x^*\in {\rm Sol}(K,F)\}$,
   \begin{align*}
    &\sqrt{2}(1-c)V(x,y)^{\frac{1}{2}}+\frac{2^\zeta(1-c)^{2\zeta}}{\sqrt{2}(1+c)} V(x,y)^{\zeta-\frac{1}{2}}\leq \rho_{\lambda,\zeta}(x,y),\nonumber\\
    &\rho_{\lambda,\zeta}(x,y)\leq \sqrt{2}(1+c)V(x,y)^{\frac{1}{2}}+\frac{2^\zeta (1+c)^{2\zeta}}{\sqrt{2}(1-c)}V(x,y)^{\zeta-\frac{1}{2}},
\end{align*}
with $c=c_\lambda$, and $\rho_{\lambda,\zeta}$ and $V$ are defined in~\eqref{eq: def_rho_lambda} and~\eqref{eq: def_V}, respectively.
\end{lem}
\begin{proof}
\an{
By Assumption~\ref{Assum-nu-quasi-monotone}, using the Cauchy-Schwarz inequality, we can see that for all $x\in K$,
\[\|F(x)-F(v^*_x)\|\dist(x,{\rm Sol}(K,F))\ge \eta \,\dist^2(x,{\rm Sol}(K,F)),\] 
where we use $\|x-v^*_x\|=\dist(x,{\rm Sol}(K,F))$ since $v^*_x\in S(x)$.
For $x\in K\setminus{\rm Sol}(K,F)$, by using the Lipschitz continuity of the mapping $F$, we obtain $L\ge \eta$. Hence, 
\[1-2\tau \eta +\tau^2 L^2=(1-\tau \eta)^2 +\tau^2 (L^2-\eta^2)\ge0.\]
We have $1-2\tau \eta +\tau^2 L^2<1$ for $\tau \in (0,\frac{2\eta}{L^2})$,
implying that $c_\tau=\sqrt{1-2\tau \eta +\tau^2 L^2}<1$.
}

\noindent
(a) 
\an{Let $x\in K$ and let $v^*_x \in S(x)$ be the projection of $x$ on ${\rm Sol}(K,F)$ satisfying Assumption~\ref{Assum-nu-quasi-monotone}.
By the non-expansiveness property of the projection, we have for any $\tau>0$,
    \begin{align*}
\|\Pi_K[x-\tau F(x)]-\Pi_K[v^*_x-\tau F(v^*_x)]\|^2&\leq \|(x-\tau F(x))-(v^*_x-\tau F(v^*_x))\|^2\cr
&\le\|x-v^*_x\|^2 -2\tau\langle F(x)-F(v^*_x),x-v^*_x\rangle 
        +\tau^2\|F(x)-F(v^*_x)\|^2.
    \end{align*}
By using Assumption~\ref{Assum-nu-quasi-monotone} and 
the Lipschitz continuity of the mapping $F$, we obtain
    \begin{align*}
\|\Pi_K[x-\tau F(x)]-\Pi_K[v^*_x-\tau F(v^*_x)]\|^2&\leq (1 +\tau^2L^2)\|x-v^*_x\|^2
        -2\tau\eta\,\dist^2(x,{\rm Sol}(K,F))\cr
        &=(1-2\tau\eta+\tau^2L^2)\dist^2(x,{\rm Sol}(K,F)),
    \end{align*}
    where the equality is due to $\|x-v^*_x\|=\dist(x,{\rm Sol}(K,F))$.
Letting $c_\tau=\sqrt{1-2\tau\eta +\tau^2L^2}$, the stated relation follows.}

\noindent
(b)
\an{From the definition of $S_\lambda(x,y)$ in~\eqref{eq: def_S_lambda} and Lemma~\ref{lem-auxiliary-natural map solution}, we have  
$V(x,y)=0$ if and only if $S_\lambda(x,y)=0$.
By the definition of $\rho_{\lambda,\zeta}$ in~\eqref{eq: def_rho_lambda}, we conclude that  
\[V(x,y)=0\quad\iff\quad \rho_{\lambda,\zeta}(x,y)=0.\]
Noting that $V(x,y)=0$ if and only if $x=x^*$, $y=x^*$, $x^*\in{\rm Sol}(K,F)$,
we have 
$V(x,y) >0$ and $\rho_{\lambda,\zeta}(x,y)>0$
for all $(x,y)\in \mathbb{R}^m\times \mathbb{R}^m\setminus\{(x^*,x^*)\mid x^*\in {\rm Sol}(K,F)\}$.
Thus, the quantity $V(x,y)^{\zeta-\frac{1}{2}}$ in the given inequalities is well defined.}

\an{Let $(x,y)\in \mathbb{R}^m\times \mathbb{R}^m\setminus\{(x^*,x^*)\mid x^*\in {\rm Sol}(K,F)\}$ be arbitrary. By Lemma~\ref{lem-auxiliary-natural map solution}, we have
$F_{K,\lambda}^{\rm nat}(x^*)=0$. In particular, by letting $x^*=v^*_x\in {\rm Sol}(K,F)$, where $v^*_x$ is from Assumption~\ref{Assum-nu-quasi-monotone}, we obtain
\begin{align}\label{eq-222}
&\|F_{K,\lambda}^{\rm nat}(x)\|
=\|F_{K,\lambda}^{\rm nat}(x)-F_{K,\lambda}^{\rm nat}(v^*_x)\|=\|x-v^*_x-(\Pi_K[x-\lambda F(x)]-\Pi_K[v_x^*-\lambda F(v_x^*)])\|,\quad
\end{align}
where the last equality is obtained
by using the definition of $F_{K,\lambda}^{\rm nat}$.
Using the triangle inequality and the fact that $\|x-v_x^*\| =\dist(x, {\rm Sol}(K,F))$, we further obtain 
\begin{eqnarray}\label{eq-111}
\dist(x, {\rm Sol}(K,F))-\|\Pi_K[x-\lambda F(x)]
    -\Pi_K[v_x^*-\lambda F(v_x^*)]\|\le \|F_{K,\lambda}^{\rm nat}(x)\|.
\end{eqnarray}
By part (a), where $\tau=\lambda$, we have
\[ -c\, \dist(x,\!{\rm Sol}(K,F))\!\le\! -\|\Pi_K[x-\lambda F(x)]-\Pi_K[v_x^*-\lambda F(v_x^*)]\|
,\]
and by adding $\dist(x, {\rm Sol}(K,F))$ to both sides of the preceding relation, we obtain
\begin{align*}
    (1-c)&\dist(x, {\rm Sol}(K,F))
    \leq \dist(x, {\rm Sol}(K,F))-\!\|\Pi_K[x-\lambda F(x)]\!-\!\Pi_K[v^*_x-\lambda F(v^*_x)]\|.
\end{align*}
Combining the preceding relation with~\eqref{eq-111}, we arrive at
\[(1-c)\dist(x, {\rm Sol}(K,F))\le 
\|F_{K,\lambda}^{\rm nat}(x)\|.\]
To upper bound $\|F_{K,\lambda}^{\rm nat}(x)\|$, we use~\eqref{eq-222} to obtain
\begin{align*}
\|F_{K,\lambda}^{\rm nat}(x)\|
    &\leq \|x-v^*_x\| +\|\Pi_K[x-\lambda F(x)]-\Pi_K[v^*_x-\lambda F(v^*_x)]\|\cr
    &\le (1+c)\dist(x, {\rm Sol}(K,F)), 
\end{align*}
where the last inequality is obtained by using $\|x-v^*_x\| =\dist(x, {\rm Sol}(K,F))$ and part~(a).
The preceding two relations yield for any $x\in K$, 
\begin{align*}
    &(1-c)^2\dist^2(x, {\rm Sol}(K,F))\leq \|F_{K,\lambda}^{\rm nat}(x)\|^2, \cr
    &\|F_{K,\lambda}^{\rm nat}(x)\|^2
    \leq (1+c)^2\dist^2(x, {\rm Sol}(K,F)).
\end{align*}
By adding $(1-c)^2\|x-y\|^2\leq\|x-y\|^2$ and 
$\|x-y\|^2\leq (1+c)^2\|x-y\|^2$ to the preceding relations, in order,  and using the definitions of $S_\lambda(x,y)$ and $V(x,y)$ (\eqref{eq: def_S_lambda} and~\eqref{eq: def_V}), we obtain
\[2(1-c)^2 V(x,y)\leq S_\lambda(x,y)\leq 2(1+c)^2V(x,y).\]
Since $\rho_{\lambda,\zeta}(x,y)=S_\lambda(x,y)^{\frac{1}{2}}+S_\lambda(x,y)^{\zeta-\frac{1}{2}}$,  it follows that 
\begin{align*}
    &\sqrt{2}(1-c)V(x,y)^{\frac{1}{2}}\!+\!\frac{2^\zeta(1-c)^{2\zeta}}{\sqrt{2}(1+c)}\,V(x,y)^{\zeta-\frac{1}{2}}\!\leq \rho_{\lambda,\zeta}(x,y),\nonumber\\
    &\rho_{\lambda,\zeta}(x,y)\!\leq \sqrt{2}(1+c)V(x,y)^{\frac{1}{2}}\!+\!\frac{2^\zeta (1+c)^{2\zeta}}{\sqrt{2}(1-c)}\,V(x,y)^{\zeta-\frac{1}{2}}.
\end{align*}
}
\end{proof}

We next prove
the existence of a unique solution to
the dynamical system~\eqref{eq: ct Korpelevich-v2}.



{\color{black}
\begin{lem}
Let Assumption \ref{Assum-nu-quasi-monotone} hold, $K\subseteq \mathbb{R}^m$ be a closed convex set, $F:K\to\mathbb{R}^m$ be a Lipschitz continuous mapping with a constant $L$, $\rho_{\lambda,\zeta}$ be as defined in \eqref{eq: def_rho_lambda}, with $\zeta\in (0, \frac{1}{2})$ and $\lambda\in (0, \frac{2\eta}{L^2})$. Then, for each initial condition $(x(0), y(0))\in K\times K$, there exists a unique solution to \eqref{eq: ct Korpelevich-v2} for all $t\geq 0$. 
\end{lem}
\begin{proof}
The dynamics in \eqref{eq: ct Korpelevich-v2} is locally Lipschitz continuous on the set $K\times K \setminus\{(x^*,x^*)\mid x^*\in\text{Sol}(K, F)\}$, since the function $\rho_{\lambda,\zeta}$ is locally Lipschitz continuous everywhere except on the set $\{(x^*,x^*)\mid x^*\in\text{Sol}(K, F)\}$.
Now, we restrict our attention to the vicinity of the set $\{(x^*,x^*)\mid x^*\in\text{Sol}(K, F)\}$. In this region, we have an upper bound on the norm of $\left((\Pi_K[x-\alpha  F(x)]-y),(\Pi_K[x-\alpha  F(y)]-x) \right)$ in terms of $\hat\theta V^{\frac{1}{2}}(x,y)$ for some $\hat\theta>0$. 
To see this, let $S(x)$ be the set of projection points of $x$ on the set 
${\rm Sol}(K,F)$, i.e., 
\[S(x)=\{v\in {\rm Sol}(K,F)\mid \|x-v\|={\rm dist}\left(x,{\rm Sol}(K,F)\right)\},\] 
and consider $x^*\in S(x)$.
Then, we have 
\begin{align*}
    \|\Pi_K[x-\alpha F(x)]-y\|=&\|\Pi_K[x-\alpha F(x)]-y+x-x\|\\
    \leq &\|\Pi_K[x-\alpha F(x)]-x\|+\|x-y\|\\
    \leq &\|\Pi_K[x-\alpha F(x)]-x-(\Pi_K[x^*-\alpha F(x^*)]-x^*)\|+\|x-y\|\\
    \leq& \|\Pi_K[x-\alpha F(x)]-\Pi_K[x^*-\alpha F(x^*)]\|+\|x-x^*\|+\|x-y\|\\
    \leq &(2+\alpha L){\rm dist}\left(x,{\rm Sol}(K,F)\right) +\|x-y\|,
\end{align*}
where in the last inequality we used the non-expansiveness property of the projection operator.
Additionally, we have 
\begin{align*}
    \|\Pi_K[x-\alpha F(y)]-x\|=&\|\Pi_K[x-\alpha F(x)+(\alpha F(x)-\alpha F(y))]-x\\
    &- \left(\Pi_K[x^*-\alpha F(x^*)]-x^*\right\|\\
    \leq&\|\Pi_K[x-\alpha F(x)+(\alpha F(x)-\alpha F(y))]\\
    &- \Pi_K[x^*-\alpha F(x^*)]\|+\|x-x^*\|\\
    \leq &\|x-\alpha F(x)-(x^*-\alpha F(x^*))\|\\
    &+\|\alpha F(x)-\alpha F(y)\|
    +\|x-x^*\|\\
    \leq &(2+\alpha L){\rm dist}\left(x,{\rm Sol}(K,F)\right)+\alpha L\|x-y\|,
\end{align*}
where in the second inequality, we used the non-expansiveness of the projection operator and the triangle inequality.
As a result, we can write
\begin{align}\label{eq:cont-argument-1}
    \|(\Pi_K[x-\alpha F(x)]-y,\Pi_K[x-\alpha F(y)]-x)\|^2&=\|\Pi_K[x-\alpha F(x)]-y\|^2+\|\Pi_K[x-\alpha F(y)]-x\|^2\nonumber\\
    &\leq ((2+\alpha L)^2+1)({\rm dist}^2\left(x,{\rm Sol}(K,F)\right)+\|x-y\|^2)\nonumber\\
    &+((2+\alpha L)^2+\alpha^2L^2)({\rm dist}^2\left(x,{\rm Sol}(K,F)\right)+\|x-y\|^2)\nonumber\\
    &\leq \hat\theta^2 V(x,y),
\end{align}
with $\hat\theta=\left[2\left(2(2+\alpha L)^2+1+\alpha^2L^2)\right)\right]^{\frac{1}{2}}$. We used Cauchy-Schwarz inequality in the derivation of the first inequality in the chain of inequalities in \eqref{eq:cont-argument-1}.
Hence, 
\begin{align*}
    &\|(\Pi_K[x-\alpha F(x)]-y,\Pi_K[x-\alpha F(y)]-x)\|\leq \hat\theta V^{\frac{1}{2}}(x,y).
\end{align*}
Combining the preceding relation and the upper bound on $\rho_{\lambda,\zeta}(x,y)$ from Lemma~\ref{Lem-aux-1}(b), we see that the norm of the system dynamic in~\eqref{eq: ct Korpelevich-v2} is upper-bounded by a linear combination of $V(x,y)$ and $V^\zeta(x,y)$. 
Moreover, by the upper bound in \eqref{eq:uperbound v by dist}, for all $x,y\in\mathbb{R}^m$, 
\begin{align*}
    V(x, y)\leq \frac{3}{2}\left({\rm dist}^2((x,y),\{(x^*,x^*)\mid x^*\in {\rm Sol}(K,F) \})\right).
\end{align*}
Since $\zeta\in (0, \frac{1}{2})$, as the distance of any point $(x,y)$ to the set $\{(x^*,x^*)\mid x^*\in\text{Sol}(K, F)\}$ goes to zero, the system dynamics also goes to zero, thus proving continuity. 

Despite the continuity of the dynamics in \eqref{eq: ct Korpelevich-v2}, it is not Lipschitz continuous everywhere, due to the multiplicative factor $\rho_{\lambda,\zeta}$. Thus, there is a solution for each initial condition. However, the standard uniqueness results cannot be used. 

Since the dynamics in \eqref{eq: ct Korpelevich-v2} is Lipschitz continuous on $K\times K \setminus\{(x,x)\mid x\in\text{Sol}(K, F)\}$, for each initial condition $(x(0), y(0))\in K\times K \setminus\{(x^*,x^*)\mid x^*\in\text{Sol}(K, F)\}$, there exists a unique solution on $[0, \tau_1)$ for some $\tau_1>0$. Now, the Lyapunov function $V$ used in the proof of Theorem~\ref{thm-Korpelevich-analysis} satisfies the conditions of \cite[Theorem 3.15.1]{agarwal1993uniqueness} for the system \eqref{eq: ct Korpelevich-v2} since the derivative of $V$ remains negative when the dynamics in \eqref{eq: ct Korpelevich} is scaled with a non-negative function $\rho_{\lambda,\zeta}$ to obtain \eqref{eq: ct Korpelevich-v2}. As a result, there exists a unique solution to \eqref{eq: ct Korpelevich-v2} on $[0, \tau_2)$ for some $\tau_2>0$. Finally, for any initial condition $(x(0), y(0))$, the level-set $\Omega = \{(x, y)\;|\; V(x, y)\leq V(x(0), y(0))\}$ is compact and positively invariant for \eqref{eq: ct Korpelevich} since the derivative of the function $V$ is non-positive. When the dynamics is modified to \eqref{eq: ct Korpelevich-v2} with a non-negative scalar multiplier $\rho_{\lambda,\zeta}$, the sign of the derivative of the function $V$ along \eqref{eq: ct Korpelevich-v2} remains non-positive, and the set $\Omega$ is positively invariant for \eqref{eq: ct Korpelevich-v2} as well. Hence, using \cite[Chapter 2, Theorem 1]{aubin1084differential}, we have that $\tau_1 = \tau_2=\infty$, i.e., a unique solution exists for all $t\geq 0$ and for any initial condition. 
\end{proof}
}

Now, we study the FxTS of the set $\{(x^*,x^*)\mid x^*\in{\rm Sol}(K,F)\}$. For this, we represent the system~\eqref{eq: ct Korpelevich-v2} as follows:

\begin{subequations}\label{eq: ct Korpelevich-v3}
    \begin{align}
        &\dot y  = G_1(x,y),  \qquad \dot x  = G_2(x,y),\\
        & G_1(x,y)=\rho_{\lambda,\zeta}(x,y)(\Pi_K[x-\alpha  F(x)]-y),\\
        & G_2(x,y) = \rho_{\lambda,\zeta}(x,y)(\Pi_K[x-\alpha  F(y)]-x),
    \end{align}
\end{subequations}
with $x(0)$,$y(0)\in K$ and $\rho_{\lambda,\zeta}$ as given in~\eqref{eq: def_rho_lambda}. 
\an{Let $z=[y^\top,x^\top]^\top$, $G(z)=[G_1(z)^\top,G_2(z)^\top]^\top$,
so
$\dot z=G(z)$ and
\[\partial V(z)=\{(\xi_y,\xi_x)\mid \xi_y\in \partial_y V(z), \    \xi_x\in \partial_x V(z)\},\]
\[\mathcal{L}_G V(z)=\{\langle\xi_x,\dot x\rangle+\langle\xi_y,\dot y\rangle \mid \xi_x\in \partial_x V(z), \ \xi_y\in \partial_y V(z)\}.\]
}

The following result establishes the fixed-time stability of the set $\{(x^*,x^*)\mid x^*\in {\rm Sol}(K,F)\}$ for the 
system in~\eqref{eq: ct Korpelevich-v3}. 
\begin{thm} \label{thm-Korpelevich-v3-analysis}
    Let $K\subseteq \mathbb{R}^m$ be a closed convex set and 
    $F:K\to\mathbb{R}^m$ be a Lipschitz continuous mapping with a constant $L>0$. Let Assumption~\ref{Assum-nu-quasi-monotone} hold. Also, let $0<\alpha <\min\{\frac{\eta}{L^2},\frac{1}{2L}\}$, $\lambda \in (0,\frac{2\eta}{L^2})$, and $0<\zeta<\frac{1}{2}$. Then, the set $\{(x^*,x^*)\mid x^*\in {\rm Sol}(K,F)\}$ is FxTS for the Korpelevich system in~\eqref{eq: ct Korpelevich-v3}, initialized at any  $x(0),y(0)\in K$, 
    with the following upper bound on the settling time
{\color{black}\[
\sup_{ x(0),y(0)\in K}\!\! T(x(0),y(0))\le \!\!
\frac{1}{\sqrt{2}d(1-c)}+\frac{2^{\frac{1}{2}-\zeta}(1+c)}{d(1-c)^{2\zeta}(1-2\zeta)},\] 
where $d=\min\left\{\alpha(\eta-\alpha L^2),\frac{1}{2}-\alpha L\right\}$ and $c=\sqrt{1-2\lambda\eta+\lambda^2L^2}$.}
\end{thm}
\begin{proof}
The solution set {\rm Sol}$(K,F)$ is nonempty by Assumption~\ref{Assum-nu-quasi-monotone},
and it is closed since the set $K$ is closed and the mapping $F$ is continuous on $K$.
Since $\alpha L<1$, by Lemma~\ref{lem:taj-in K-sys-2} the
equilibrium set of~\eqref{eq: ct Korpelevich-v3} coincides with the set {\rm Sol}$(K,F)$. 

Define a candidate Lyapunov function $V:\mathbb R^m\times\mathbb R^m\to [0,+\infty)$ by: for all $x,y\in \mathbb{R}^m$,
\begin{align*}
    V(x, y) = \frac{1}{2}\left({\rm dist}^2(x,{\rm Sol}(K,F)) + \|x - y\|^2\right).
\end{align*}
We have $V(x,y)=0$ for any $x\in {\rm Sol}(K,F)$ and $x=y$. Using the triangle inequality, we can obtain for all $x,y$,
\begin{align*}
    {\rm dist}^2(x,{\rm Sol}(K,F))+{\rm dist}^2(y,{\rm Sol}(K,F))\leq 6 V(x, y),
\end{align*}
implying that condition~(1) of Lemma~\ref{Lem-2-Kunal's paper} holds. Next, we show that condition (2) of Lemma~\ref{Lem-2-Kunal's paper} is satisfied.
Let $S(x)$ be the set of projection points of $x$ on the set 
${\rm Sol}(K,F)$, i.e., 
\[S(x)=\{v\in {\rm Sol}(K,F)\mid \|x-v\|={\rm dist}\left(x,{\rm Sol}(K,F)\right)\}.\]
We have 
$\nabla_y V(x,y)=y-x,$
while by Lemma~\ref{lem-sudiff-dist2} (since $V$ is locally Lipschitz) we have
\begin{align*}
 \partial_x\! V(x,y) &= x-y+\!\left\{\begin{array}{ll}
    \!\! \! x-{\rm conv}(S(x)) &  \hbox{if $x\notin {\rm Sol}(K,F)$},\cr
\!\! \! \an{\{0\}}& \hbox{otherwise}.
\end{array}\right.   
\end{align*}
For the Lie derivative of $V$ along the trajectories of \eqref{eq: ct Korpelevich-v2}, using the structure of $\partial V$, we have
\[\mathcal{L}_G V(x, y)=\{\langle u,\dot x\rangle +\langle y-x,\dot y\rangle \mid u\in\partial_xV(x,y)\},
\]
implying that
\begin{align*}
   \max\mathcal{L}_G V(x, y)=\langle x-y,\dot x-\dot y\rangle + \max\left\{\!\!\!
\begin{array}{ll}
    \{\langle x-v,\dot x\rangle|\, v\in {\rm conv}(S(x))\} & \!\!\! \hbox{if $x\notin {\rm Sol}(K,F)$},\cr
    \an{\{0\}}
    & \!\!\!\hbox{otherwise}.
\end{array}\right.   
\end{align*}
Considering the linearity with respect to $v\in \conv(S(x))$, by \eqref{eq-milman}, taking the maximum over $v\in \conv(S(x))$ is the same as taking the maximum over $v\in S(x)$. 
Hence, we obtain 
\begin{align*}
   \max\mathcal{L}_G V(x, y)=\langle x-y,\dot x-\dot y\rangle& + \max\left\{\!\!\!
\begin{array}{ll}
    \{\langle x-v,\dot x\rangle|\, v\in S(x)\} & \!\!\! \hbox{if $x\notin {\rm Sol}(K,F)$},\cr
    \an{\{0\}}
    & \!\!\!\hbox{otherwise}.
\end{array}\right.   
\end{align*}

Now, we consider two cases separately, the case when $x\notin {\rm Sol}(K,F)$, and the case when $x\in {\rm Sol}(K,F)$ but $y\ne x$.

\noindent
{\it Case $x\notin {\rm Sol}(K,F)$.}
For the Lie derivative, we have
\begin{align}\label{eq-liederiv}
    \max\mathcal{L}_G V(x, y)=&\langle x-y,\dot x-\dot y\rangle+
    \max\{\langle x-v,\dot x\rangle|\, v\in S(x)\}.
    \end{align}
Using the system dynamic~\eqref{eq: ct Korpelevich-v3}, we obtain
\begin{align*}
    \langle x-y,\dot x-\dot y\rangle &= -\rho_{\lambda,\zeta}(x,y)\|x-y\|^2 +\rho_{\lambda,\zeta}(x,y)\langle x-y, \Pi_K[x-\alpha F(y)]-\Pi_K[x-\alpha F(x)]\rangle\cr
    &\le 
    -\rho_{\lambda,\zeta}(x,y)\|x-y\|^2 +\rho_{\lambda,\zeta}(x,y)\|x-y\|\,\|\Pi_K[x-\alpha F(y)]-\Pi_K[x-\alpha F(x)]\|.
\end{align*}
By using the non-expansiveness property of the projection mapping $\Pi_K$ and the Lipschitz continuity of $F$, we obtain
\begin{align}\label{eq-liederiv-p1}
    \langle x-y,\dot x-\dot y\rangle \le -(1-\alpha L)\rho_{\lambda,\zeta}(x,y)\|x-y\|^2.
\end{align}

Now, we consider the set on the right hand side in~\eqref{eq-liederiv}. For an arbitrary $v\in S(x)$, we have
\begin{align*}
    \langle x-v,\dot x\rangle &=\rho_{\lambda,\zeta}(x,y)
    \langle x-v, \Pi_K[x-\alpha F(y)]-x\rangle.
\end{align*}
Since $S(x)\subset {\rm Sol}(K,F)$, the point $v$ is a solution to VI($K,F)$. Thus, by Lemma~\ref{lem-auxiliary-natural map solution},   it satisfies $v-\Pi_K[v-\alpha F(v)]=0$, and by using this equality in the preceding relation, we obtain
\begin{align*}
    \langle x-v,\dot x\rangle &=\rho_{\lambda,\zeta}(x,y)\langle x-v, v-\Pi_K[v-\alpha F(v)] \rangle +\rho_{\lambda,\zeta}(x,y)\langle x-v,\Pi_K[x-\alpha F(y)]-x\rangle\cr
    & = \rho_{\lambda,\zeta}(x,y)\langle x-v, \Pi_K[x-\alpha F(y)] -\Pi_K[v-\alpha F(v)]\rangle  -\rho_{\lambda,\zeta}(x,y)\|x-v\|^2.
\end{align*}
Hence,
\begin{align*}
    \langle x-v,\dot x\rangle 
     &\le -\rho_{\lambda,\zeta}(x,y)\|x-v\|^2 +\rho_{\lambda,\zeta}(x,y)\|x-v\|\, \|\Pi_K[x-\alpha F(y)] -\Pi_K[v-\alpha F(v)]\|\cr 
    &\le -\frac{\rho_{\lambda,\zeta}(x,y)}{2}\|x-v\|^2+ \frac{\rho_{\lambda,\zeta}(x,y)}{2}\|\Pi_K[x-\alpha F(y)] -\Pi_K[v-\alpha F(v)]\|^2,
\end{align*}    
where the last inequality follows from $ab\le (a^2+b^2)/2$ valid for any scalars $a,b$.
To bound $\|\Pi_K[x-\alpha F(y)] -\Pi_K[v-\alpha F(v)]\|^2$, we use  Lemma~\ref{Lem-aux-1}(a) with $\tau=\alpha$, and obtain
\begin{align*}
    \langle x-v,\dot x\rangle 
    & \le \frac{(c_\alpha^2 -1)\rho_{\lambda,\zeta}(x,y)}{2}\|x-v\|^2 =-\left(\alpha\eta-\frac{\alpha^2 L^2}{2}\right)\rho_{\lambda,\zeta}(x,y) 
    \|x-v\|^2.
\end{align*}   
Finally, using  $\|x-v\|=\dist(x,{\rm Sol}(K,F))$, we obtain
\begin{equation}\label{eq-liederiv-p2}
\langle x-v,\dot x\rangle 
    \le -\left(\alpha\eta-\frac{\alpha^2 L^2}{2}\right)\rho_{\lambda,\zeta}(x,y) \dist^2(x,{\rm Sol}(K,F)).
    \end{equation}
Using relations~\eqref{eq-liederiv-p1} and~\eqref{eq-liederiv-p2}
in equation~\eqref{eq-liederiv}, we conclude that every element in $\mathcal{L}_G V(x, y)$ is upper bounded by 
\begin{align*} 
  &- \rho_{\lambda,\zeta}(x,y) (1-\alpha L)\|x-y\|^2-\rho_{\lambda,\zeta}(x,y)\left(\alpha\eta-\frac{\alpha^2 L^2}{2}\right)\dist^2(x,{\rm Sol}(K,F)).
    \end{align*}
Thus, for the maximum over $v\in S(x)$, we have that
\begin{align}\label{eq-liederiv2} 
&\max\mathcal{L}_G V(x, y) \le 
- \rho_{\lambda,\zeta}(x,y) (1-\alpha L)\|x-y\|^2-\rho_{\lambda,\zeta}(x,y)\left(\alpha\eta-\frac{\alpha^2 L^2}{2}\right)\dist^2(x,{\rm Sol}(K,F)).
  \end{align}
Hence,
\begin{align*}
\max\mathcal{L}_G V(x, y)\leq -\rho_{\lambda,\zeta}(x,y)\left(\alpha\eta-\frac{\alpha^2L^2}{2}\right) {\rm dist}(x,{\rm Sol}(K,F))^2 - \rho_{\lambda,\zeta}(x,y)\left(\frac{1}{2}-\alpha L\right)\|x-y\|^2 .
\end{align*}

We have
$d=\min\left\{\alpha\eta-\alpha^2L^2,\frac{1}{2}-\alpha L\right\}>0$ for $0<\alpha<\min\{\frac{\eta}{L^2},\frac{1}{2L}\}$, 
and using the Lyapunov function $V$, we have
\begin{align}\label{eq-liederiv3} 
    \max \mathcal{L}_G V(x, y)\leq -2d\rho_{\lambda,\zeta}(x,y) V(x,y).
\end{align}
Moreover, using the lower bound on $\rho_{\lambda,\zeta}(x,y)$ from Lemma~\ref{Lem-aux-1}(b), we obtain
\begin{align} \label{eq-liederiv4}
    \max \mathcal{L}_G V(x, y)\leq& -2\sqrt{2}d(1-c) V(x,y)^{\frac{3}{2}}-2d\frac{2^\zeta (1-c)^{2\zeta} V(x,y)^{\frac{1}{2}+\zeta}}{\sqrt{2}(1+c)}.
\end{align}

\noindent
{\it Case $x\in {\rm Sol}(K,F)$ and $y\ne x$.} In this case, the Lie derivative is given by
$\mathcal{L}_GV(x,y) =\langle x-y,\dot x-\dot y\rangle$. 
Using relation~\eqref{eq-liederiv-p1}, and noting that $V(x,y)=\|x-y\|^2/2$ when $x\in{\rm Sol}(K,F)$,
we find that relation \eqref{eq-liederiv3} holds. Then, using the same arguments as in the preceding case, we conclude that the inequality \eqref{eq-liederiv4} holds in this case as well.

Thus, since the function $V$ is locally Lipschitz and the solution trajectory of \eqref{eq: ct Korpelevich-v2} is unique for any initial condition, by Lemma~\ref{lem: Second-cond}, we have  $D^*V(x(t),y(t))\leq \max \mathcal{L}_GV(x, y)$. By Lemma~\ref{Lem-2-Kunal's paper}, the statement of the theorem follows.
\end{proof}

\section{Discussion}
\label{sec-disc}

\subsection{Alternative assumptions}
In this part, we provide an alternative for Assumption~\ref{Assum-nu-quasi-monotone} when a mapping $F$ is defined on $\mathbb{R}^m$.
\begin{assum}\label{Assum:locally-mu-quasi-strongly-monotone}
Assume that
 there is $\eta>0$ and an open ball $B(z^*,r)$ centered at $z^*\in\mathbb{R}^m$ with a radius $r>0$ such that 
    \[\langle F(x)-F(z^*),x-z^*\rangle\geq \eta\|x-z^*\|^2\qquad
    \hbox{for all }x\in B(z^*,r).\]
\end{assum}

Next, we relate $\eta$ of Assumption~\ref{Assum:locally-mu-quasi-strongly-monotone} to the smallest eigenvalue of $\nabla F(z^*)+\nabla F (z^*)^\top$ when $F$ is differentiable.
\begin{thm}\label{thm:smallest eigenvalue_Jacobian-direction-1}
    Let mapping $F:\mathbb{R}^m\to \mathbb{R}^m$ be continuously differentiable.  Then, the following statements are valid:
    \begin{itemize}
        \item [a)] If Assumption~\ref{Assum:locally-mu-quasi-strongly-monotone} holds, then the smallest eigenvalue of $\frac{\nabla F(z^*)+\nabla F (z^*)^\top}{2}$ is lower bounded by $\eta$, i.e., \[\lambda_{min}\left(\frac{\nabla F(z^*)+\nabla F (z^*)^\top}{2}\right)\geq\eta>0.\]
        \item [b)] If $\lambda_{min}\left(\frac{\nabla F(z^*)+\nabla F (z^*)^\top}{2}\right)\geq\nu>0$ at some $z^*\in\mathbb{R}^m$, then $F$ satisfies Assumption~\ref{Assum:locally-mu-quasi-strongly-monotone} with $\eta=\frac{\nu}{2}$. 
    \end{itemize}
\end{thm}
\begin{proof} (a) Let  $w\in\mathbb{R}^m$ be arbitrary nonzero vector and let $z(t)=z^*+tw$ for $t>0$ small enough so that 
$z(t)\in B(z^*,r)$. 
By the intermediate value theorem, for every $i=1,\ldots,m$, 
there is some $s_i\in (0,t)$ such that 
    \begin{align*}
        F_i(z(t))-F_i(z^*)
        &=\langle \nabla F_i(z^*+s_i(z(t)-z^*)),z(t)-z^* \rangle= \langle \nabla F_i(z^*+s_itw),z(t)-z^*\rangle. 
    \end{align*}
    Letting $z_i=z^*+s_itw$ for each $i$, we have for all $i=1,\ldots,m$,
    \begin{align}\label{eq-1-smallest eigenvalue_Jacobian}
        F_i(z(t))-F_i(z^*)
        = \langle \nabla F_i(z^i),z(t)-z^*\rangle. 
    \end{align}
    Let
   ${\bf z}_m=\{z^1,\ldots,z^m\}$ be the collection of the points $z_i=z^*+s_itw$. Construct a matrix $\nabla_{{\bf z}_m} F$ from the Jacobians $\nabla F(z^1),\ldots,\nabla F(z^m)$, by taking the $i$-th row of each Jacobian $\nabla F(z^i)$, $i=1,\ldots,m$, i.e.,
\[\nabla F_{{\bf z}_m} =\left[\begin{array}{c}
[\nabla F(z^1)]_{1:}\cr\vdots\cr
[\nabla F(z^m)]_{m:}
\end{array}\right],\]
with $A_{i:}$ denoting the $i$-th row of a matrix $A$.
Using this matrix, we concatenate all equalities in~\eqref{eq-1-smallest eigenvalue_Jacobian} and obtain
\begin{equation*}
        F(z(t))-F(z^*)=\nabla F_{\bf z_m}\cdot(z(t)-z^*).
    \end{equation*}
    From the preceding relation and Assumption~\ref{Assum:locally-mu-quasi-strongly-monotone}, since $z(t)\in B(z^*,r),$ we obtain
    \begin{align*}
        \langle \nabla F_{\bf z_m}\!\!\cdot\!(z(t)\!-\!z^*),z(t)-z^*\rangle
        &=\langle F(z(t))-F(z^*), z(t)\!-\!z^*\rangle\geq \eta \|z(t)-z^*\|^2.
    \end{align*}
    Using the fact $z(t)-z^*=tw$, with $t>0$, we further obtain
    \[\langle \nabla F_{\bf z_m}\!\cdot w,w\rangle
    \geq \eta \|w\|^2.\]
    As $t\to 0$, the collection ${\bf z_m}$ reduces to a collection of $m$ copies of $z^*$, and $\nabla F_{\bf z_m}$ tends to $\nabla F(z^*)$. Thus, by letting $t\to 0$, we have
    \[ \frac{1}{2}\langle (\nabla F(z^*) +\nabla F(z^*)^\top) w,w\rangle
    =\langle \nabla F(z^*) w,w\rangle
    \geq \eta \|w\|^2,\]
    implying that all eigenvalues of $\frac{\nabla F(z^*)+\nabla F (z^*)^\top}{2}$ are bounded from below by $\eta>0$. 

\noindent(b) The proof is along the lines of the proof of \cite[Proposition 2.3.2, part c)]{facchinei2003finite} restricted to a local ball containing $z^*$. 
\end{proof}

The following example illustrates that Assumption~\ref{Assum:locally-mu-quasi-strongly-monotone} need not hold for a dynamic with an FxTS equilibrium. 
\begin{exmp}\label{exmp-relation-mu-quasi-str-FxTS}
Let \[
A =
\begin{bmatrix}
1 & 10 \\
0 & 1
\end{bmatrix},\qquad M=\frac{1}{2}(A+A^\top)=\begin{bmatrix}
1 & 5 \\
5 & 1
\end{bmatrix}.\]
The eigenvalues of $A$ are $\{1,1\}$, and
the eigenvalues of $M$ are $\{6,-4\}$. The mapping $x\mapsto Ax$ 
does not satisfy Assumption~\ref{Assum:locally-mu-quasi-strongly-monotone} by Theorem~\ref{thm:smallest eigenvalue_Jacobian-direction-1}(a). However, by~\cite[Theorem~4.6 and Corollary~4.3]{khalil2002nonlinear}, the origin is exponentially stable for the dynamic $\dot{x}=-Ax$. 
Consider a Lyapunov function $V(x)=x^\top Px$ with a positive definite $P$ such that $(-A^\top) P+P(-A)=-Q$ for some positive definite $Q$ (such matrices $P,Q$ exist by \cite[Corollary~8.2.1]{DATTA2004245}).
 Let \textcolor{black}{$\rho(x)=V(x)^{\frac{1}{2}}+V(x)^{-\frac{1}{2}+\zeta}$ with $\zeta\in(0,\frac{1}{2})$}, and consider the system $\dot{x}=-\rho(x)Ax$. We have $\dot{V}(x)\leq -\rho(x)\lambda_{\min}(Q)\|x\|^2$. Moreover, for the Lyapunov function $V$, we have
 $\|x\|^2\ge \lambda^{-1}_{\max}(P)\, V(x)$. Hence, \textcolor{black}{
 \begin{align*}
     \dot{V}(x)& \leq -\rho(x)\lambda_{\min}(Q)\lambda^{-1}_{\max}(P)\,V(x)\leq -\lambda_{\min}(Q)\lambda^{-1}_{\max}(P) \left(V(x)^\frac{3}{2}+V(x)^{\frac{1}{2}+\zeta}\right).
 \end{align*}} Thus, for this system, 
the origin is FxTS with a settling time bounded from above by \textcolor{black}{$\frac{4(1-\zeta)}{(1-2\zeta)}\lambda_{\max}(P)\lambda_{\min}^{-1}(Q)$}.
\end{exmp}

The next example shows that
$\nabla F$ may have an eigenvalue with a negative real part (violating Assumption~\ref{Assum:locally-mu-quasi-strongly-monotone}), yet there might be some dynamics with FxTS equilibrium guarantee.

\begin{exmp}\label{exmp-relation-with-eigenvalues}
    Consider a mapping $F:\mathbb{R}^{4}\to\mathbb{R}^{4}$ given by $F(x)=[Ax_1+2Bx_2,3C x_1+Dx_2]$, where $x_1,x_2\in \mathbb{R}^2$ and $A,B,C,D$ are symmetric $2\times 2$ matrices. 
Let
\begin{align*}
    &A = \begin{bmatrix}-1 & 0.5 \\ 0.5 & 2\end{bmatrix},\quad
B = \begin{bmatrix}0.2 & 0.1 \\ 0.1 & 0.3\end{bmatrix},\\
&C = \begin{bmatrix}0.4 & 0 \\ 0 & 0.1\end{bmatrix},\quad
D = \begin{bmatrix}3 & 0.2 \\ 0.2 & 4\end{bmatrix}.
\end{align*}
For this choice, the Jacobian $\nabla F(x)$ is: for all $x\in\mathbb{R}^4$,
\[
\nabla F (x)= 
\begin{bmatrix}
A & 2B \\[2mm]
3C & D
\end{bmatrix}
=
\begin{bmatrix}
-1 & 0.5 & 0.4 & 0.2 \\
0.5 & 2 & 0.2 & 0.6 \\
1.2 & 0 & 3 & 0.2 \\
0 & 0.3 & 0.2 & 4
\end{bmatrix}.
\]
The eigenvalues of $\nabla F(x)$ (rounded to two decimal points) are $\{-1.18, 4.15, 3.07, 1.96\},$
so $\nabla F(x)$ has both positive and negative real eigenvalues. 
The eigenvalues of $(\nabla F(x)+\nabla F(x)^\top)/2$ (rounded to two decimal points) are $\{-1.22, 4.16, 3.11, 1.95\},$  
so $F$ does not satisfy Assumption~\ref{Assum:locally-mu-quasi-strongly-monotone}. 
The eigenvalues of $\nabla F(x)\nabla F(x)^{\top}$ are $\{17.38, 10.20, 3.84, 1.28\},$
 so $\nabla F(x)\nabla F(x)^{\top}$ is positive definite; hence $\nabla F(x)$ is invertible. By satisfying the conditions in Theorem~\ref{Thm:Design-continuous-dynamic-fixed-time-stable}, the equilibria of the system~\eqref{eq:first-fxTS-sys} are FxTS.
\end{exmp}

\subsection{Discrete-time Variant}
{\color{black}We provide a brief discussion on the following discrete-time variant of the continuous-time Korpelevich system in~\eqref{eq: ct Korpelevich-v2}, for some $\gamma>0$, 
\begin{subequations}\label{eq: dt Korpelevich-v2}
    \begin{align}
         y^{k+1}\!\! & = y^k\!\!+\gamma\rho_{\lambda,\zeta}(x^k,y^k)\left( \Pi_K[x^k\!-\alpha  F(x^k)]-y^k\right),\\
         x^{k+1}\!\! & = x^k\!\!+\gamma\rho_{\lambda,\zeta}(x^k,y^k)\left( \Pi_K[x^k\!-\alpha  F(y^k)]-x^k\right),
    \end{align}
\end{subequations}
initiated with $x^0,y^0\in K$. In the absence of additional conditions, an upper bound for $\rho_{\lambda,\zeta}(x,y)$ is not available, thus not allowing us to
specify a range for $\gamma$ so that the iterates of  the method in~\eqref{eq: dt Korpelevich-v2} remain in the set $K$. To deal with this, we will assume\footnote{An alternative is to extend the mapping outside the set $K$ by letting $F(x):=F(\Pi_K[x])$ for all $x\notin K$.} that $\dom(F)=\mathbb{R}^m$. The work in \cite{garg2022fixed} establishes convergence properties for the iterates of the method to reach an $\varepsilon$-neighborhood of the point $(x^*,x^*)$, where $x^*$ is the unique solution of VI($K,F$). In particular, it is shown in \cite[Theorem 2]{garg2022fixed} that the trajectories of the discretized dynamics \eqref{eq: dt Korpelevich-v2} remain within an $\epsilon-$neighborhood for all times and reach the $\epsilon-$neighborhood of the equilibrium of the continuous-time dynamics \eqref{eq: ct Korpelevich-v2} within a fixed number of steps. We refer the interested reader to \cite{garg2022fixed} for more details. 
}


\section{Simulations}\label{Sec: Simulations}
We consider the method in~\eqref{eq: dt Korpelevich-v2} for solving VI($K,F)$ with $F(x)=[5x_1+x_2,x_1+5x_2]$ and $K=[-10,10]\times [-10,10]$. The eigenvalues of $\left(\nabla F(x)+\nabla F(x)\right)^\top/2$ are $\{6, 4\}$ for all $x\in K$.
The unique solution of VI$(K,F)$ is $x^*=[0,0]^\top$. We have $\langle F(x)-F(x^*),x-x^*\rangle\geq 4\|x-x^*\|^2$; hence, have $\eta=4$ and $L= 6$. We let $\lambda = \frac{\eta}{L^2}$, $\alpha  = \lambda$, and $\gamma\in\{0.02,0.04,0.06,0.08\}$, and $x^0=[8,-6]^\top$, $y^0=[-5,7]^\top$, $\zeta=0.2$. As seen in Figure~\ref{fig:FxTS1}, increasing $\gamma$ \sa{in \eqref{eq: dt Korpelevich-v2}} 
results in a faster convergence. Figure~\ref{fig:FxTS2} shows the effect of the $\gamma$ parameter in the contraction of the mapping associated with the algorithm, i.e., the decrease of $\|x^{k+1}-x^k\|$ with the number $k$ of iterations. 

\begin{figure}[h!]
    \centering
    \begin{minipage}{0.52\textwidth}
        \centering
\includegraphics[width=\linewidth]{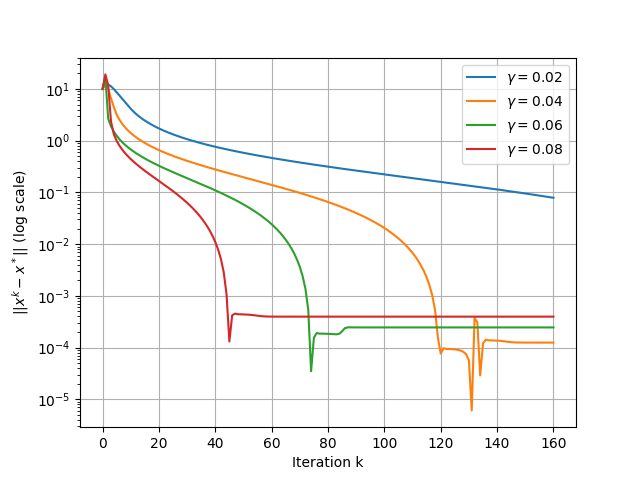}
        \caption{Log-scale convergence of the Korpelevich-type algorithm~\eqref{eq: dt Korpelevich-v2}.}
        \label{fig:FxTS1}
    \end{minipage}\hfill
    \begin{minipage}{0.52\textwidth}
        \centering
        \includegraphics[width=\linewidth]{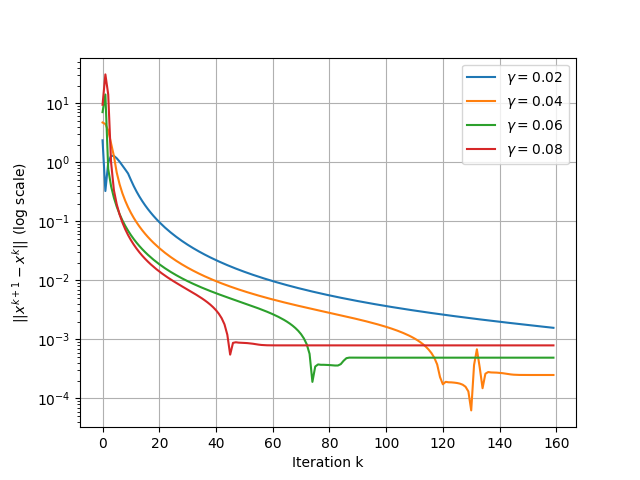}
        \caption{Log-scale of the norm  $\|x^{k+1}-x^k\|$ for different $\gamma$ values. }
        \label{fig:FxTS2}
    \end{minipage}
\end{figure}

\section{Conclusion}\label{Sec: Conclusion}
In this paper, we have investigated exponential and fixed-time stability of the solution set of a class of non-monotone VIs. We have combined tools from non-smooth and Lyapunov analysis to develop stability results for solutions of unconstrained and constrained VIs under the Korpelevich dynamic. Through examples, we have demonstrated that different assumptions can be used to assert FxTS of equilibria for a dynamical system. We supported our results with simulations that demonstrate the behavior of the Korpelevich dynamic.

{\color{black}Although a brief discussion on the discretized dynamics and its convergence behavior was provided based on the existing results, more work needs to be done for finite- or fixed-time convergence of discrete time algorithms for solving VIs. Future work will explore further relaxations of the conditions imposed in this work to accelerate convergence through designing discrete-time variants with similar convergence properties.}

\bibliographystyle{plain}
\bibliography{refs.bib}

\appendix

\section{Proofs of Auxiliary Results}
\subsection{Proof of Lemma~\ref{Lem-2-Kunal's paper}}\label{Sec:Proof of Lem-2-Kunal's paper}
\begin{proof}
The equilibrium set $M$ is closed due to our underlying assumption that $f$ in~\eqref{Eq:Sys} is continuous.  The global finite-time stability of the set $M$ will follow by the assumption on the growth of $V$ stated in condition~(1) of the lemma, provided that $V(x(\hat t))=0$ at some finite time $\hat t\ge0$ for any initial point $x(0)\in{\rm dom}(f)$, which we prove next.

Let $x(\cdot)$ be the trajectory of the system in~\eqref{Eq:Sys} with an arbitrary initial point $x(0)\in{\rm dom}(f)$.
Along this trajectory, define $\psi(s)=V(x(s))$ for all $s\ge0$. Since $V$ is locally Lipschitz continuous and $x(\cdot)$ is differentiable, the scalar function $\psi$ is also locally Lipschitz continuous. Thus, by Rademacher's  Theorem~\ref{The:Rademacher}, the function $\psi$ is differentiable almost everywhere on $[0,+\infty)$, i.e., $\dot \psi(s)=D^*\psi(s)$ for almost all $s\ge0$, except for a set of zero Lebesgue measure. Therefore, for all $t\ge0$,
\begin{equation}\label{eq:app_rad_th}
     \int_{0}^{t}\frac{\dot \psi(s)}{\psi^{\beta\theta}(s)}ds=\int_{0}^{t}
    \frac{D^{*}V(x(s))}{V^{\beta\theta}(x(s))}ds.
\end{equation}

By condition (2) we have 
\begin{equation}\label{eq:lem-1-2}
    D^{*}V(x(t)) \le -q^{\theta} V^{\beta\theta}(x(t)) \quad \text{for } V(x(t)) > 1,
\end{equation}
\begin{equation}\label{eq:lem-1-1}
    D^{*}V(x(t)) \le -p^{\theta} V^{\alpha\theta}(x(t)) \quad \text{if } V(x(t)) \le 1.
\end{equation}

Given an initial point $x(0)\in{\rm dom}(f)$, there are two possibilities:  $V(x(0)) > 1$ or $V(x(0)) \le 1$, which are considered separately.

\noindent
{\it Case $V(x(0)) > 1$}. Suppose that the trajectory $x(\cdot)$ starts at $x(0)\in{\rm dom}(f)$ with $V(x(0)) > 1$.
Then, by~\eqref{eq:lem-1-2} for an interval $[0,t]$ where $V(x(s)) > 1$ for all $s\in[0,t]$, we have 
\begin{align}\label{eq:lem-1-2-1}
    -\frac{V^{1-\beta\theta}(x(t))}{\beta\theta-1}+\frac{V^{1-\beta\theta}(x(0))}{\beta\theta-1}
    =\int_{0}^{t}\frac{\dot \psi(s)}{\psi^{\beta\theta}(s)}ds,
    \end{align}
    where we used $\beta\theta>1$.
    Therefore, from~\eqref{eq:lem-1-2-1} we obtain
\begin{align}\label{eq:lem-1-2-2}
    -\frac{V^{1-\beta\theta}(x(t))}{\beta\theta-1}+\frac{V^{1-\beta\theta}(x(0))}{\beta\theta-1}    
    &=\int_{0}^{t}
    \frac{D^{*}V(x(s))}{V^{\beta\theta}(x(s))}ds \leq -\int_{0}^{t}q^{\theta}ds=-q^{\theta} t,
\end{align}
where the inequality follows from~\eqref{eq:lem-1-2}.
Hence, 
\[\frac{V^{1-\beta\theta}(x(t))}{q^{\theta}(\beta\theta-1)}\geq  \frac{V^{1-\beta\theta}(x(0))}{q^{\theta}(\beta\theta-1)}+t > t,\]
implying that
\[\frac{1}{t q^{\theta}(\beta\theta-1)}\ge V^{\beta\theta-1}(x(t)).\]
As $t$ increases, the value $V^{\beta\theta-1}(x(t))$ decreases, so there must exist the first time $t_1$ when $V(x(t_1))=1$, implying that 
\[\frac{1}{q^{\theta}(\beta\theta-1)}\ge t_1,\]
and for any $t\ge t_1$, we have $V(x(t))\le 1.$

Next, for any $t\ge t_1$, using~\eqref{eq:lem-1-1} and a line of analysis similar to the derivation of~\eqref{eq:lem-1-2-2}, we can show that
\begin{align*}
    \frac{V^{1-\alpha\theta}(x(t))}{1-\alpha\theta}
    -\frac{V^{1-\alpha\theta}(x(t_1))}{1-\alpha\theta} 
    \leq -\int_{t_1}^{t}p^{\theta}ds=-p^{\theta} (t-t_1).
\end{align*}
Thus,
\begin{equation}\label{eq:lem-1-3}
    \frac{V^{1-\alpha\theta}(x(t))}{p^{\theta}(1-\alpha\theta)}-\frac{V^{1-\alpha\theta}(x(t_1))}{p^{\theta}(1-\alpha\theta)}\le - (t-t_1),
\end{equation}
implying that
\begin{align*}
 \frac{V^{1-\alpha\theta}(x(t))}{p^{\theta}(1-\alpha\theta)} 
 &\le   \frac{V^{1-\alpha\theta}(x(t_1))}{p^{\theta}(1-\alpha\theta)} -(t-t_1)\le \frac{1}{p^{\theta}(1-\alpha\theta)} -(t-t_1),   
\end{align*}
where the last inequality follows from $V(x(t_1))= 1$ and $1-\alpha\theta>0$. As $t$ increases, the value $V(x(t))$ decreases, so there is the first time $t_2>t_1$ when $V(x(t_2))=0$, implying that
\[\frac{1}{p^{\theta}(1-\alpha\theta)} -(t_2-t_1)\ge0.\]
Since $V(x(t_2))=0$, it follows that $\dist(x(t_2),M)=0$ by condition~(1), implying that $x(t_2)\in M$ since $M$ is closed. Also, since $V(x)=0$ when $x\in M$ by condition~(1), it follows that $x(t)\in M$ for all $t\ge t_2$.
Thus, the time $t_2$ is the settling time $T(x(0))$ for the dynamic and it satisfies
\[T(x(0))\le\frac{1}{p^{\theta}(1-\alpha\theta)} +t_1\le \frac{1}{p^{\theta}(1-\alpha\theta)}+\frac{1}{q^{\theta}(\beta\theta-1)}.\]

\noindent
{\it Case $V(x(0)) \le 1$}.
Suppose that a trajectory $x(\cdot)$ starts at $x(0)\in{\rm dom}(f)$ with $V(x(0)) \le 1$ and $V(x(0))>0$. In this case, the preceding analysis from time $t_1$ to time $t_2$ applies with $t_1=0$. Hence, by relation~\eqref{eq:lem-1-3} with $t_1=0$, we have
\begin{equation}\label{eq-llast}
\frac{V^{1-\alpha\theta}(x(t))}{p^{\theta}(1-\alpha\theta)} 
 \le   \frac{V^{1-\alpha\theta}(x(0))}{p^{\theta}(1-\alpha\theta)} -t.
 \end{equation}
 As $t$ increases, the quantity on the right-hand side of inequality~\eqref{eq-llast} decreases. Thus, there exists the first time $\hat t>0$, when $V(x(\hat t))=0$. Similar to the preceding case, by using condition~(1) we conclude that
 $x(\hat t)\in M$, and  $x(t)\in M$ for all $t\ge \hat t$.
 Thus, the time $\hat t$ is the settling time $T(x(0))$ in this case, and letting $t=T(x(0))$  in~\eqref{eq-llast} yields
 \[T(x(0))\le \frac{V^{1-\alpha\theta}(x(0))}{p^{\theta}(1-\alpha\theta)} \le 
 \frac{1}{p^{\theta}(1-\alpha\theta)},\]
 where the last inequality is obtained by using $V(x(0)) \le 1$ and \hbox{$1-\alpha\theta>0$}. This completes the proof.
\end{proof}

 \subsection{Proof of Lemma~\ref{lem: Second-cond}}\label{Sec:lem: Second-cond}
 \begin{proof}
 Let $t\ge0$ be an arbitrary but fixed time and $s>0$.
We write $x(t+s)=x(t+s)-sf(x(t))+sf(x(t))$ so that 
\[x(t+s)=v(s) + sf(x(t)) \quad\hbox{with} \quad
v(s)=x(t+s)-s f(x(t)).\]
Then, we have
\begin{align*}
    D^*V(x(t))
    &=\limsup_{ s\to 0^+}\frac{V(x(t+s))-V(x(t))}{s}=\limsup_{ s\to 0^+}\frac{V(v(s)+sf(x(t)))-V(x(t))}{s}.
    \end{align*}
    By adding and subtracting $V(v(s))$ in the numerator of the preceding relation, 
    we obtain
    \begin{align*}
    D^*V(x(t))
    &\le\limsup_{ s\to 0^+}\frac{V(v(s)+sf(x(t)))-V(v(s))}{s} +\limsup_{ s\to 0^+}\frac{V(v(s))-V(x(t))}{s}.\cr
\end{align*}
Since $\lim_{s\to 0^+}v(s)=x(t)$ by continuity of $x(\cdot)$, we have
\begin{align*}
&\limsup_{ s\to 0^+}\frac{V(v(s)+sf(x(t)))-V(v(s))}{s}\le \limsup_{y\to x(t), s\to0^+}\frac{V(y+s f(x(t))-V(y)}{s}.\end{align*}
As a result, 
\begin{align*}
    D^*V(x(t))
    &\leq\limsup_{y\to x(t), s\to0^+}\frac{V(y+s f(x(t))-V(y)}{s} +\limsup_{ s\to 0^+}\frac{V(v(s))-V(x(t))}{s}.
\end{align*}
Using the definition of the Clarke generalized directional derivative, we have
\begin{align}\label{eq:dir1}
    D^*V(x(t))
    &\leq V^\circ(x(t);f(x(t))) +\limsup_{ s\to 0^+}\frac{V(v(s))-V(x(t))}{s}.
\end{align}
By the local Lipschitz continuity of $V(x)$, we have 
\begin{align}\label{eq:dir2}
\limsup_{ s\to 0^+}\frac{V(v(s))-V(x(t))}{s} &\le 
    \limsup_{ s\to 0^+}\frac{|V(v(s))-V(x(t))|}{s}\leq L_{t} \limsup_{ s\to 0^+}\frac{\|v(s)-x(t)\|}{s},
\end{align}
where $L_t$ is a local Lipschitz constant for the point $x(t)$.
Since $v(s)=x(t+s)-sf(x(t))$ and $x(\cdot)$ is a solution to \eqref{Eq:Sys} 
(satisfies~\eqref{eq:sys-solution}), we have
\[v(s)-x(t)=\int_{t}^{t+s}f(x(s))ds-s f(x(t)),\]
implying that 
\begin{align*}
    \lim_{s\to 0^+}\frac{\|v(s)-x(t)\|}{s}
    &=
\lim_{s\to 0^+} \left\|\frac{1}{s}\int_{t}^{t+s}\!f(x(s))ds- sf(x(t))\right\|=0.
\end{align*}
The preceding relation and~\eqref{eq:dir1}--\eqref{eq:dir2} yield
\[D^*(V(x(t))\le V^o(x(t);f(x(t))).\]
As $t\ge0$ is arbitrary and $V^{\circ}(x(t);f(x(t)))=\max \mathcal{L}_fV(x(t))$ for all $t\ge0$ 
by \cite[Proposition~1.5(b)]{clarke1998nonsmooth}, the result follows.
\end{proof}

\subsection{Proof of Lemma~\ref{lem-sudiff-dist2}}\label{ap-proof-subdiff}
\begin{proof}
By \cite[Example 8.53] {rockafellar1998variational}, the subdifferential set of the distance function ${\rm dist}(\cdot,U)$ is given by
\[\partial {\rm dist} (z,U)=\left\{
\begin{array}{ll}
\left\{\frac{z-v}{{\rm dist}(z,U)}\mid v\in U_z\right\} & \hbox{if $z\notin U$},\cr
N_U(z)\cap \mathbb{B}(0,1) & \hbox{if $z\in U$},
\end{array}\right. \]
where $U_z$ is the compact subset of $U$ consisting of the points that attain the minimum distance from $z$ to the set $U$, i.e., 
$U_z=\{v\in U\mid \|z-v\|={\rm dist}(z,U)\},$
the set $N_U(z)$ is the normal cone of the set $U$ at the point $z\in U$ (\cite[Definition 6.3]{rockafellar1998variational},
and 
$\mathbb{B}(0,1)$ is the closed unit ball centered at the origin.
Applying the chain rule \cite[Theorem 2.3.9(ii)]{clarke1990optimization} to a composition of the form $g\circ h$ (which applies since $g(s)=s^2$ is continuously differentiable, hence strictly differentiable~\cite[Corollary, page 32]{clarke1990optimization}), the subdifferential set $\partial \, {\rm dist}^2(z,U)$ is given by
\begin{align}\label{eq:clarke-1}
\partial{\rm dist}^2(z,U)\!=\!\!\left\{\begin{array}{ll}
    \!\! \! 2\,{\rm conv}\left(\{z-v \mid v\in U_z\}\right) &\!\! \hbox{if $z\notin U$},\cr
\!\! \! \{0\} & \!\!\hbox{if $z\in U$},
\end{array}\right.
\end{align}
where  $tS$ is a scaling of the set $S$ by a scalar $t$ and  ${\rm conv}(S)$ is the convex hull of a set $S$. \an{In~\eqref{eq:clarke-1}, we have $\partial \, {\rm dist}^2(z,U)=\{0\}$ when $z\in U$,  since
${\rm dist} (z,U)=0$.}

Using the Minkowski sum notation, i.e., $X+Y=\{x+y\mid x\in X, y\in Y\}$, for the subdifferential in~\eqref{eq:clarke-1}
we have
\[\partial\, {\rm dist}^2(z,U)=2{\rm conv}\left(\{z\}-U_z\right)
\qquad\hbox{for $z\notin U$}.\]
The Minkowski sum and the convex hull operations commute
\cite[Theorem~1.1.2]{Schneider_2013}, i.e., ${\rm conv}(X+Y)={\rm conv}(X)+{\rm conv}(Y)$ for any two sets $X,Y\subseteq\mathbb{R}^m$. Moreover, ${\rm conv}(tX)=t\,{\rm conv(X)}$ for any set $X$ and any scalar $t$, so it follows that  
\[\partial\, {\rm dist}^2(z,U)=2\left(\{z\}-{\rm conv}(U_z)\right)
\qquad\hbox{for $z\notin U$}.\]
The result follows from the preceding relation and~\eqref{eq:clarke-1}.
\end{proof}

\end{document}